\documentclass[english]{article}
\usepackage[utf8]{inputenc}
\usepackage[T1]{fontenc}
\usepackage{lmodern}
\usepackage[a4paper,left=2cm,right=2cm,top=2cm,bottom=2cm]{geometry}
\usepackage{amsmath}
\usepackage{amssymb}
\usepackage{amsthm}
\usepackage{babel}
\usepackage{color}
\usepackage{multirow}
\usepackage{tikz}
\usepackage{algorithmic,algorithm2e,float}
\definecolor{erin}{RGB}{0,255,63}
\definecolor{orchid}{RGB}{218,112,214}
 \def\dpol{\mathop{{}^{\textsc{D}}\text{Pol}}}

\newcommand{\bA}{\mathbf{A}}
\newcommand{\bB}{\mathbf{B}}
\newcommand{\bC}{\mathcal{C}}
\newcommand{\bF}{\mathcal{F}}
\newcommand{\bM}{\mathcal{M}}
\newcommand{\bN}{\mathcal{N}}
\newcommand{\bS}{\mathcal{S}}
\newcommand{\bU}{\mathbb{U}}
\newcommand{\Comb}{\mathrm{Comb}}
\newcommand{\codim}[1]{\mathrm{codim}(#1)}
\newcommand{\fracexp}[3]{\left(\frac{#1}{#2}\right)^{\!\!#3}}
\newtheorem{theorem}{Theorem}[section]
\newtheorem{lemma}[theorem]{Lemma}
\newtheorem{corollary}[theorem]{Corollary}
\newtheorem{proposition}[theorem]{Proposition}
\theoremstyle{definition}
\newtheorem{definition}[theorem]{Definition}

\theoremstyle{remark}
\newtheorem*{remark}{Remark}
\author{No\'emie Combe\footnote{Aix-Marseille Univ, CNRS, Centrale Marseille, I2M, Marseille, France}~~and Vincent Jug\'e\footnote{
LSV, CNRS, ENS Paris-Saclay \& Inria, France --- This author is supported by EU under ERC EQualIS (FP7-308087).}}
\title{Counting bi-colored A'Campo forests}
\begin{document}

\maketitle

\abstract{ We give an efficient method of counting bi-colored A'Campo forests and provide a stratification criterion for the space of those graphs. 
We deduce an algorithm to count these forests in polynomial time, answering an open question of A'Campo.}

\section{Introduction}
\subsection{Results on the number of A'Campo forests}
We consider the problem suggested by A'Campo of counting the number of special bi-colored forests, recently introduced in~\cite{AC16.2}. 
These forests are combinatorial objects given by two systems of $d$ planar curves colored in $O$ and in $E$ respectively,
which are the inverse images of the real and imaginary axis by the polynomial map $P$.
In order to suit our considerations we draw our forests in a disk, with terminal vertices on the boundary of the disk and call them {\it configurations}. 

We give an efficient method of counting these graphs and deduce a computer program to count these forests in polynomial time.
In addition, an asymptotic counting for large degrees is given. In order to make the counting effective we introduce the notion of canonical splitting and partial configuration. 

Precisely, the number of A'Campo forests follows from the following theorem:

\bigskip

{\bf Theorem}~\ref{thm:19}
\textit{
Let us denote by $\#N_1(c,d)$ is the number of configurations of codimension $c$ and degree $d$,
and let $\bN_1$ be the associated bivariate generating function, defined by: $\bN_1(x,y) = \sum_{c,d \geq 0} \#N_1(c,d) x^c y^d$.
There exists an auxiliary bivariate generating function $\bN_2$, with non-negative coefficients, such that $\bN_1(x,y)$ and $\bN_2(x,y)$ are solutions of:
\[\bN_1 = 1 + y \bN_2^4 \text{ and } (1 + y \bN_2^4 - x y \bN_2^5)(1 - \bN_2 + 2 y \bN_2^4 - y \bN_2^5 + x y \bN_2^6 + y^2 \bN_2^8) + x^3 y^2 \bN_2^{11} = 0.\]
}
 
\subsection{Stratifying polynomial forests}
This counting problem raises the question of defining a stratification for the space of graphs.  
Consider the space $\dpol_{d}$ of monic, complex polynomial maps of degree $d>1$, $P:\mathbb{C}\to \mathbb{C}$ with distinct roots. 

These decorated graphs verify the seven following properties~\cite{AC16.2}:
 
\begin{enumerate}
\item  The graph has no cycles. The graph is a forest. The non-compact edges are properly embedded in $\mathbb{C}$.
\item The complementary regions  have a $4$-colouring by symbols $A,B,C,D$.
\item  The edges are oriented and have a $2$-colouring  by symbols $E,O$. The symbol $E$ if the edge
separates $D$ and $A$ or $B$ and $C$ coloured regions.  The symbol $O$ if the edge
separates $A$ and $B$ or $C$ and $D$ coloured regions. The orientation is right-handed if one crosses the edge from
$D$ to $A$ or $A$ to $B$, and left-handed if one crosses $B$ to $C$ or $C$ to $D$.  
\item  The picture has $4d$ edges that are near infinity asymptotic to the 
rays $re^{k\pi \imath/4d},\,r>0,\,k=0,1, \cdots ,4d-1$. 
The colors $R,I$ alternate and the orientations
of the $E$ coloured and also the $O$ coloured alternate between out-going and in-going.
\item  Near infinity the sectors are coloured in the counter-clockwise orientation by the $4$-periodic sequence of symbols $A,B,C,D,A,B, \cdots $.
\item  The graph can have $5$ types of vertices:
for the first $4$ types only $A,B$ or $B,C$ or $C,D$ or $D,A$ regions are incident and only edges of one color are incident, moreover for the fifth type regions of all $4$ colors are incident and the colors appear in the counter clock-wise orientation as $A,B,C,D,A,B, \cdots $. So, in particular the graph has no terminal vertices.
\item  At all points $p\in \pi_P$ the germ of graph $\pi_P$ is for a $k=1,2, \cdots$ 
smoothly diffeomorphic to the germ at $0\in \mathbb{C}$ of $\{z \in \mathbb{C} \mid {\rm Re}(z^k)=0\}$. 
 \end{enumerate}
These graphs are classified by attributing a number $d$ and a number $c$ which are respectively the crossing number and the codimension $c$. 
The crossing number is given by the number of intersections of $E$ and $O$ curves. These are the roots of the degree $d$ polynomial $P$, hence this number is equal to $d$. 
 
On the other hand, the codimension number $c$ is given by the number and multiplicity of intersecting curves of the same color, such that the crossing number remains equal to $d$. 
Indeed, these points have multiplicity higher than 1 and are thus the critical points of the maps $Re(P)$ and $Im(P)$.
In particular, by critical point we mean that at a given point $(x,y)\in \mathbb{R}^2$ the partial derivatives of the maps are zero and satisfy the algebraic equation $Re(P)=0$ and/or $Im(P)=0$.  

Our stratification of the space of graphs relies on the number and multiplicties of these critical points.
For instance, a codimension $c=0$ configuration contains no meeting points and $d$ crossing points.
The configurations of codimension $c=1$ are those such that there exists one critical point for $Re(P)$ or $Im(P)$ of multiplicity 2:
this point is drawn as the intersection of two curves of the same color.
Two critical points for the configurations contribute to defining a forest of codimension $c=2$. 
 
Roughly speaking, the codimension $c$ is defined by the number of inner nodes which have incident edges of the same color and by their valencies. 
The codimension number is $c=\sum_{i=1}^{m}c_{i}$ where the $c_{i}$ are the multiplicities of the critical points.
We adopt the following rule concerning the counting of the codimensions:
if the valency of the inner node is $2k$, then the multiplicity of the critical point is $k$, and its local codimension is equal to $2k-1$.
Indeed, a critical point is of multiplicity $k$ when the derivative polynomial $P'$ is constrained to have a root of multiplicity $k-1$,
itself located on the union of varieties of codimension 1 formed by the space $\{z \in \mathbb{C} \mid {\rm Re}(P(z)) = 0 \text{ or } {\rm Im}(P(z)) = 0\}$.
Finally, observe that, since the polynomials we consider have distinct roots, the crossing points of their graphs are all distinct.
 
\section{Configurations and splittings}
\label{section:1}
In order to count our graphs, we draw them in a disk with the terminal vertices on its boundary. Such drawings are called \emph{configurations}.
We are particularly interested in counting configurations up to some equivalence relation:
two configurations are considered equivalent if there exists a homeomorphism of the disk, leaving the boundary of the disk invariant,
and that maps one configuration to another.
The notion of configuration will be redefined so as to suit our next considerations. Showing that this definition matches the one above is straightforward.
In addition, we introduce the notion of splitting of a configuration, which will be necessary for our the counting result.

\begin{definition}
\label{def:1}
Let $d$ be a non-negative integer. 
Let $P_0, P_1, \ldots, P_{4d-1}$ be points lying in the clockwise order on the unit circle $\partial \bU$.
In addition, let us consider $2d$ piecewise-affine lines lying inside the unit disk $\bU$,
which we divide between \emph{odd} and \emph{even} lines, so that:
\begin{itemize}
 \item every point $P_a$ belongs to one line, and every line touches $\partial \bU$ at its endpoints only;
 \item every odd (respectively, even) line joins points $P_a$ and $P_b$ such that $\{a,b\} \equiv \{0,2\} \pmod{4}$ (respectively, $\{a,b\} \equiv \{1,3\} \pmod{4}$);
 \item every line touches one line of the opposite parity, at a point that we call \emph{crossing point}, and it must cross that line at that point;
 \item two lines of the same parity may touch each other at some point, which we call \emph{meeting point}, and they may not cross each other at that point;
 \item the union of these $2d$ lines is cycle-free.
\end{itemize}
The union of these $2d$ lines is called a \emph{configuration} of degree $d$, and the
collection of these $2d$ lines is called a \emph{splitting} of the configuration.
\end{definition}
Note that a configuration $\bA$ may have several splittings.
This phenomenon is illustrated in Figure~\ref{fig:1} (even lines are drawn in \textbf{\color{erin}erin}, and odd lines are drawn in \textbf{\color{orchid}orchid}).
Section~\ref{section:2} is devoted to defining a canonical splitting of $\bA$.
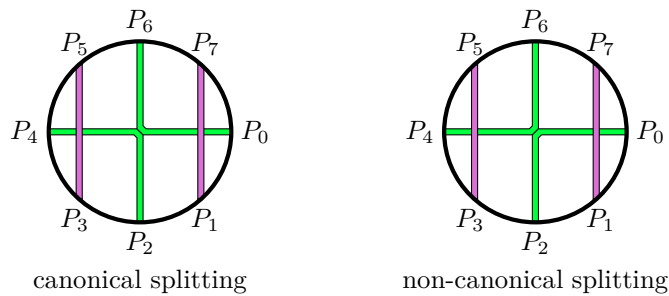
\begin{figure}[!ht]
\begin{center}
\begin{tikzpicture}[scale=0.4,>=stealth]
\draw[fill=erin] (-3,0.1) -- (-0.1,0.1) -- (0.1,-0.1) -- (0.1,-3) -- (-0.1,-3) -- (-0.1,-0.2) -- (-0.2,-0.1) -- (-3,-0.1) -- cycle;
\draw[fill=erin] (3,-0.1) -- (0.1,-0.1) -- (-0.1,0.1) -- (-0.1,3) -- (0.1,3) -- (0.1,0.2) -- (0.2,0.1) -- (3,0.1) -- cycle;
\draw[fill=orchid] (-1.9,-2.24) -- (-1.9,2.24) -- (-2.1,2.14) -- (-2.1,-2.14) -- cycle;
\draw[fill=orchid] (1.9,-2.24) -- (1.9,2.24) -- (2.1,2.14) -- (2.1,-2.14) -- cycle;
\draw[ultra thick] (0,0) circle (3);
\draw[fill=erin] (10,-0.1) -- (12.9,-0.1) -- (13.1,0.1) -- (13.1,3) -- (12.9,3) -- (12.9,0.2) -- (12.8,0.1) -- (10,0.1) -- cycle;
\draw[fill=erin] (16,0.1) -- (13.1,0.1) -- (12.9,-0.1) -- (12.9,-3) -- (13.1,-3) -- (13.1,-0.2) -- (13.2,-0.1) -- (16,-0.1) -- cycle;
\draw[fill=orchid] (11.1,-2.24) -- (11.1,2.24) -- (10.9,2.14) -- (10.9,-2.14) -- cycle;
\draw[fill=orchid] (14.9,-2.24) -- (14.9,2.24) -- (15.1,2.14) -- (15.1,-2.14) -- cycle;
\draw[ultra thick] (13,0) circle (3);
\node[anchor=west] at (3,0) {$P_0$};
\node[anchor=north] at (2.15,-2.24) {$P_1$};
\node[anchor=north] at (0,-3) {$P_2$};
\node[anchor=north] at (-2.15,-2.24) {$P_3$};
\node[anchor=east] at (-3,0) {$P_4$};
\node[anchor=south] at (-2.15,2.24) {$P_5$};
\node[anchor=south] at (0,3) {$P_6$};
\node[anchor=south] at (2.15,2.24) {$P_7$};
\node[anchor=north] at (0,-4.25) {canonical splitting};
\node[anchor=west] at (16,0) {$P_0$};
\node[anchor=north] at (15.15,-2.24) {$P_1$};
\node[anchor=north] at (13,-3) {$P_2$};
\node[anchor=north] at (10.85,-2.24) {$P_3$};
\node[anchor=east] at (10,0) {$P_4$};
\node[anchor=south] at (10.85,2.24) {$P_5$};
\node[anchor=south] at (13,3) {$P_6$};
\node[anchor=south] at (15.15,2.24) {$P_7$};
\node[anchor=north] at (13,-4.25) {non-canonical splitting};
\end{tikzpicture}
\caption{Two splittings of the same configuration (of degree $2$ and codimension $1$)}
\end{center}
\label{fig:1}
\end{figure}
Two other important notions for configurations are the \emph{codimension} and the \emph{equivalence} of configurations.
\begin{definition}
\label{def:2}
Let $\bA$ be a configuration. We define the \emph{codimension} of a meeting point of $\bA$ as follows:
\begin{itemize}
 \item a point at which $k \geq 3$ lines of $\bA$ meet each other has codimension $k$;
 \item a point at which exactly $2$ lines of $\bA$ meet each other has codimension $1$.
\end{itemize}
The codimension of $\bA$ is defined as the sum of the codimension of its meeting points.
\end{definition}
\begin{definition}
\label{def:3}
Two configurations $\bA$ and $\bB$ are said \emph{equivalent} if
some homeomorphism of the unit disk maps $\bA$ to $\bB$ and maps each point $P_i$
to itself.
\end{definition}
Note that the codimension and the degree of a configuration does not depend on which splitting we considered,
and that two equivalent configurations have the same degree and the same codimension.
Hereafter, we aim to compute the number of (equivalence classes of) configurations with codimension $c$ and degree $d$.

\section{Canonical splitting}

\label{section:2}
We saw above that some configurations may have several splittings.
Distinguishing which configurations have one splitting leads to the following definition.
\begin{definition}
\label{def:4}
A configuration $\bA$ without any meeting point is said to be a \emph{flat} configuration.
\end{definition}
Indeed, every flat configuration has one splitting.
We investigate now a way to design a canonical splitting for every configuration (including the non-flat ones).

Throughout this section we call $\bA$ a configuration, $\bU$ the unit disk and $\bC$
a connected component of $\bU \setminus \bA$. The lemmas always include these three objects.

\begin{lemma}
\label{lem:5}
There exists a (unique) integer $k \in \mathbb{Z}/4\mathbb{Z}$ such that
the border $\partial \bC \cap \partial \bU$ is a non-empty union of
arcs of circles of the form $[P_a,P_{a+1}]$ with $a \equiv k \pmod{4}$.
The integer $k$ is called the \emph{index} of $\bC$,
and is denoted by $\iota(\bC)$.
\end{lemma}
\begin{proof}
First, since $\bA$ is cycle-free, it comes at onces that
$\partial \bC \cap \partial \bU$ is non-empty.
Hence, we write it as a disjoint union $[P_{a_1},P_{a_1+1}] \cup \ldots \cup [P_{a_i},P_{a_i+1}]$ of arcs of circles,
with $a_1 < a_2 < \ldots < a_i$ (and $P_{4d} = P_0$).
If that $i \geq 2$, let $L$ be the fragment of $\partial \bC$ that joins $P_{a_1+1}$ to $P_{a_2}$.
The line $L$ splits the disk $\bU$ into two parts: one part $\bU_1$ (which excludes the line $L$)
that contains $\bC$ and one part $\bU_2$ (which contains $L$) that does not.
Let $\bS$ be a splitting of $\bA$, and let
$\bS'$ be the collection of those lines of $\bS$ that intersect the area $\bU_2$.
By construction, no line of $\bS$ may go inside the interior of $\bC'$,
and therefore all lines in $\bS'$ belong entirely to the area $\bU_2$.
If the collection $\bS'$ contains $\ell$ crossing points,
then it contains exactly $\ell$ odd and $\ell$ even lines, which join
$4\ell$ points on $\partial \bU$ overall.
This means that the set of endpoints $\{P_{a_1+1},\ldots,P_{a_2}\}$ has cardinality $4\ell$,
and therefore that $a_2 \equiv a_1 \pmod{4}$. This completes the proof.
\end{proof}
\begin{definition}
\label{def:6}
Let $L$ be a connected component of $\partial \bC \setminus \partial \bU$.
We call $L$ a \emph{diagonal} of $\bA$.
\end{definition}
\begin{lemma}
\label{lem:7}
Every diagonal of $\bA$ is a piecewise-affine line
whose endpoints are two points $P_a$ and $P_b$ lying on $\partial \bU$
and contains one crossing point of $\bA$.
Conversely, every crossing point belongs to $4$ pairwise distinct diagonals of $\bA$.
\end{lemma}
\begin{proof}
Let $L$ be a diagonal of $\bA$, let $\bS$ be a splitting of $\bA$, and let $d$ be the degree of $\bA$.
First, since $\bA$ is cycle-free, we know that $L$ is a piecewise-affine line between two points $P_a$ and $P_b$ lying on $\partial \bU$.
Without loss of generality, $a$ is even and $b$ is odd.
Second, if $L$ does not contain any crossing point of $\bA$, then, progressing on $L$ from $P_a$ to $P_b$,
we observe that $L$ consists only of fragments of even lines of $\bS$, which is impossible since $L$ ends in $P_b$.
Hence, $L$ contains at least one crossing point of $\bA$.
Finally, the configuration $\bA$ has $4d$ diagonals and $d$ crossing points,
each of which belongs to exactly $4$ diagonals (since $\bA$ is cycle-free).
Hence, every crossing point of $\bA$ belongs to exactly $4$ diagonals, which do not contain any other crossing point,
thereby completing the proof.
\end{proof}
\begin{lemma}

\label{lem:8}
Let $L$ be a connected component of $\partial \bC \setminus \partial \bU$ connecting the points 
$P_a$ and $P_b$ with$a \equiv b-1 \equiv \iota(\bC) \pmod{4}$. 
$L$ is endowed with one crossing point $X$.
Take another point $Z$ of $L$ which is neither a meeting point or a crossing point.
If $Z$ lies between $P_a$ and $X$, then $Z$ is adjacent to two components of indices $\iota(\bC)$ and $\iota(\bC)+1$;
if $Z$ lies between $X$ and $P_b$, then $Z$ is adjacent to two components of indices $\iota(\bC)-1$ and $\iota(\bC)$.
\end{lemma}
\begin{proof}
We first assume that $\bA$ is flat.
Then, the entire fragment of $L$ lying between $P_a$ and $X$ separates two components
(one that contains the arc of circle $[P_{a-1},P_a]$ and one that contains the arc of circle $[P_a,P_{a+1}]$)
with respective indices $\iota(\bC)$ and $\iota(\bC)+1$.
Similarly, the fragment of $L$ lying between $X$ and $P_b$ separates two components with respective indices $\iota(\bC)-1$ and $\iota(\bC)$,
and therefore Lemma~\ref{lem:8} holds in this case.
In the general case, observe that, for every $\varepsilon > 0$, there exists a flat configuration
$\bA_\varepsilon$ that differs from $\bA$ only on disks centered on $\bA$'s meeting points with radius $\varepsilon$.
If $\varepsilon$ is small enough, then:
\begin{itemize}
 \item $Z$ does not belong to any of these disks, hence it is adjacent to two connected components
 $\bC_1$ and $\bC_2$ of $\bU \setminus \bA_\varepsilon$ with suitable indices;
 \item every point not belonging to these disks nor to $\bA$ lies in one connected component
 $\bC_0$ of $\bU \setminus \bA$ and one connected component $\bC_\varepsilon$ of $\bU \setminus \bA_\varepsilon$,
 such that $\iota(\bC_0) = \iota(\bC_\varepsilon)$.
\end{itemize}
It follows that Lemma~\ref{lem:8} holds in the general case too.
\end{proof}
\begin{lemma}
\label{lem:9}
Every crossing point of the configuration $\bA$ is adjacent to four connected components of $\bU \setminus \bA$,
with indices $0$, $1$, $2$ and $3$.
\end{lemma}
\begin{proof}
Due to Lemma~\ref{lem:8} (when choosing points $Z$ arbitrarily close from the crossing point),
every crossing point adjacent to a connected component of $\bU \setminus \bA$
with index $k$ is also adjacent to connected components with respective indices $k-1$ and $k+1$.
The result follows.
\end{proof}
Suppose that $\bA$, $L$, $X$ are as above.
\begin{definition}
\label{def:10}

Let $P_a$ and $P_b$ be the endpoints of $L$, with $a \equiv b-1 \pmod{4}$.
The fragment of $L$ that lies between $P_a$ and $X$ is called a \emph{starting diagonal},
and the fragment of $L$ that lies between $X$ and $P_b$ is called an \emph{ending diagonal}.
\end{definition}
\begin{lemma}
\label{lem:11}
Let $\bA$ be a configuration and let $X$ be a point of $\bA$ that is neither a crossing point nor a meeting point.
The point $X$ belongs to one starting diagonal of $\bA$.
\end{lemma}
\begin{proof}
Observe that $X$ is adjacent to two connected components $\bC_1$ and $\bC_2$ of $\bU \setminus \bA$.
Without loss of generality, and due to Lemma~\ref{lem:8}, we have $\iota(\bC_2) = \iota(\bC_1)+1$.
Let $L_1$ be the connected component of $\partial \bC_1 \setminus \partial \bU$ to which belongs $X$,
and let $L_2$ be defined similarly.
Then, $X$ belongs to the starting fragment of $L_1$ and to the ending fragment of $L_2$,
hence it belongs to a unique starting diagonal of $\bA$.
\end{proof}
\begin{definition}
\label{def:12}
Let $X$ be a crossing point of $\bA$.
For all $k \in \mathbb{Z}/4\mathbb{Z}$,
let $L_k$ be the starting diagonal adjacent to $X$ with an endpoint $P_a$ such that $a \equiv k \pmod{4}$.
The two lines $L_0 \cup L_2$ and $L_1 \cup L_3$ are called the \emph{canonical splitting lines} of $X$.
\end{definition}
\begin{proposition}
\label{pro:13}
Let $\bS$ be the collection of all canonical splitting lines of all the crossing points of $\bA$.
This collection is a splitting of $\bA$, which we call the \emph{canonical splitting} of $\bA$.
\end{proposition}
\begin{proof}
First, every line $L \in \bS$ has endpoints $P_a$ and $P_b$ such that either $\{a,b\} \equiv \{0,2\} \pmod{4}$
or $\{a,b\} \equiv \{1,3\} \pmod{4}$, hence it is either odd or even.
Furthermore, $L$ is a canonical splitting line of some crossing point $X \in \mathbf{C}$,
at which it crosses another line $L' \in \bS$, with the opposite parity.
Second, let $\mathbf{C}$ be the set of all crossing points of $\bA$.
Since starting diagonals never cross each other, none of the starting diagonals $L$ is made of
can cross another line $L'' \in \bS$ at a point $Y \notin \mathbf{C}$.
Hence, $L'$ is the only line in $\bS$ that crosses $L$ (which it does at $X$ only).
Third, observe that every connected component of $\bA \setminus \mathbf{C}$
either contains only points $P_a$ with $a \equiv 0 \pmod{2}$ or
contains only points $P_a$ with $a \equiv 1 \pmod{2}$.
Consequently, whenever two starting diagonals touch each other
at a point $Y \notin \mathbf{C}$,
in addition to not crossing each other, they must have the same parity.
Finally, Lemma~\ref{lem:11} shows that $\bA$ is the union of all the lines in $\bS$.
Recalling that $\bA$ is cycle-free completes the proof.
\end{proof}

\section{Counting A'Campo forests}
\label{section:3}
Canonical splittings of configurations pave the way to recursive decompositions of configurations,
leading to the enumeration of our graphs.
However, taking into account meeting points leads us to introduce the following combinatorial objects.
\begin{definition}
\label{def:14}
Let $d$ be a non-negative integer and consider the points $P_{-1}, P_0, \ldots, P_{4d}$ on the boundary of the disk $\partial \bU$.
Take the $2d+1$ piecewise-affine lines embedded in the unit disk $\bU$.
We partition the set of lines into $d$ \emph{odd}, $d$ \emph{even} lines and one \emph{base} line. These lines verify the following conditions:
\begin{itemize}
 \item every point $P_a$ belongs to one line, and every line touches $\partial \bU$ at its endpoints only;
 \item the base line joins the points $P_{-1}$ and $P_{4d}$;
 \item every odd (respectively, even) line joins points $P_a$ and $P_b$ such that $\{a,b\} \equiv \{0,2\} \pmod{4}$ (respectively, $\{a,b\} \equiv \{1,3\} \pmod{4}$);
 \item every line touches one line of the opposite parity, at a point that we call \emph{crossing point}, and it must cross that line at that point;
 \item two lines of the same parity may touch each other at some point, which we call \emph{meeting point}, and they may not cross each other at that point;
 \item the base line may touch even lines only, and may not cross them;
 \item the union of these $2d+1$ lines is cycle-free.
\end{itemize}
The union of these $2d+1$ lines is called a \emph{partial configuration} of degree $d$, and the
collection of these $2d+1$ lines is called a \emph{splitting} of the partial configuration.
If, furthermore, the points $P_{-1}$ and $P_{4d-1}$ belong to the same connected component of the partial configuration,
then the partial configuration is said to be \emph{widespread}.
\end{definition}
Like in the case of configurations,
the \emph{codimension} of a partial configuration is the sum of the codimensions of its meeting points,
and two partial configurations $\bA$ and $\bB$ are \emph{equivalent} if some homeomorphism of the unit disk
maps $\bA$ to $\bB$ and maps each point $P_a$ to itself.
Hereafter, configurations and partial configurations are only considered up to equivalence.
We denote by $N_1(c,d)$ (respectively, $N_2(c,d)$ and $N_3(c,d)$) the set of configurations
(respectively, partial configurations and widespread partial configurations)
with codimension $c$ and degree $d$.
By extension, for $i \in \{1,2,3\}$,
we also denote by $N_i$ the set $\bigcup_{c,d \geq 0} N_i(c,d)$
and by $\bN_i$ the associated bivariate generating function, defined by:
\[\bN_i(x,y) = \sum_{c,d \geq 0} \#N_i(c,d) x^c y^d.\]
We investigate now recursive decompositions of (standard, partial) configurations,
which will give rise to equations involving the generating functions $\bN_i$.
\begin{lemma}
\label{lem:15}
We have $\bN_1 = 1 + y \bN_2^4$.
Furthermore, we can associate unambiguously every partial configuration with a splitting,
which we call \emph{canonical} splitting of this configuration.
\end{lemma}
\begin{proof}
We design a bijection $\varphi : N_1 \mapsto \{\emptyset\} \cup N_2^4$,
such that $\varphi(\bA) = \emptyset$ if $\bA$ has degree $0$,
and such that $\varphi(\bA) = (\bA_0,\bA_1,\bA_2,\bA_3)$ where
$\deg \bA = \sum_{i=0}^3 \deg \bA_i + 1$ and
$\codim{\bA} = \sum_{i=0}^3 \codim{\bA_i}$ if $\bA$ has degree at least $1$.
First, there exists one configuration $\bA$ of degree $0$, hence we safely associate it with $\emptyset$.

Then, if $\bA$ has degree at least $1$, let $\bS$ be the canonical splitting of $\bA$.
Let $L_1$ be the line of $\bS$ with endpoint $P_0$, let $L_2$ be the line of $\bS$ that crosses $L_1$,
and let $X$ be the crossing point at which $L_1$ and $L_2$ cross each other.
Observe that $L_1$ has two endpoints $P_{a_0}$ and $P_{a_2}$, and $L_2$ has two endpoints $P_{a_1}$ and $P_{a_3}$, with
and $a_k \equiv k \pmod{4}$ for all $k$. It is straightforward that $0 = a_0 < a_1 < a_2 < a_3 < a_4 = 4d$.
Furthermore, the set $\bU \setminus (L_1 \cup L_2)$ is formed of $4$ connected components $\bC_0, \ldots, \bC_3$,
where $\bC_k$ contains the arc of circle $[P_{a_k},P_{a_{k+1}}]$ on its border.
For $0 \leq k \leq 3$, we do the following.

We delete all the points $P_x$ that do not belong to the arc of circle $[P_{a_k},P_{a_{k+1}}]$
then we renumber each of the points $P_x$ (with $a_k \leq x \leq a_{k+1}$) to $P_{x-a_k-1}$.
Doing so, we observe that $\bA \cap \overline{\bC_k}$ is a partial configuration, which we denote by $\bA_k$.
Hence, we define $\varphi(\bA)$ as the tuple $(\bA_0,\bA_1,\bA_2,\bA_3)$,
as illustrated in Fig.~\ref{fig:2}.
By construction, every crossing point (beside the point $X$) and every meeting point of $\bA$
belongs to one partial configuration $\bA_k$.
Furthermore, the mapping $\varphi$ is clearly bijective.
This proves the equality $\bN_1 = 1 + y \bN_2^4$.

Furthermore, let $\bB$ be a partial decomposition and let $\bB_\emptyset$ be the unique partial decomposition of degree $0$.
The canonical splitting of
$\varphi^{-1}(\bB,\bB_\emptyset,\bB_\emptyset,\bB_\emptyset)$ induces a splitting of $\bB$,
which is the above-mentioned canonical splitting of $\bB$.
\end{proof}
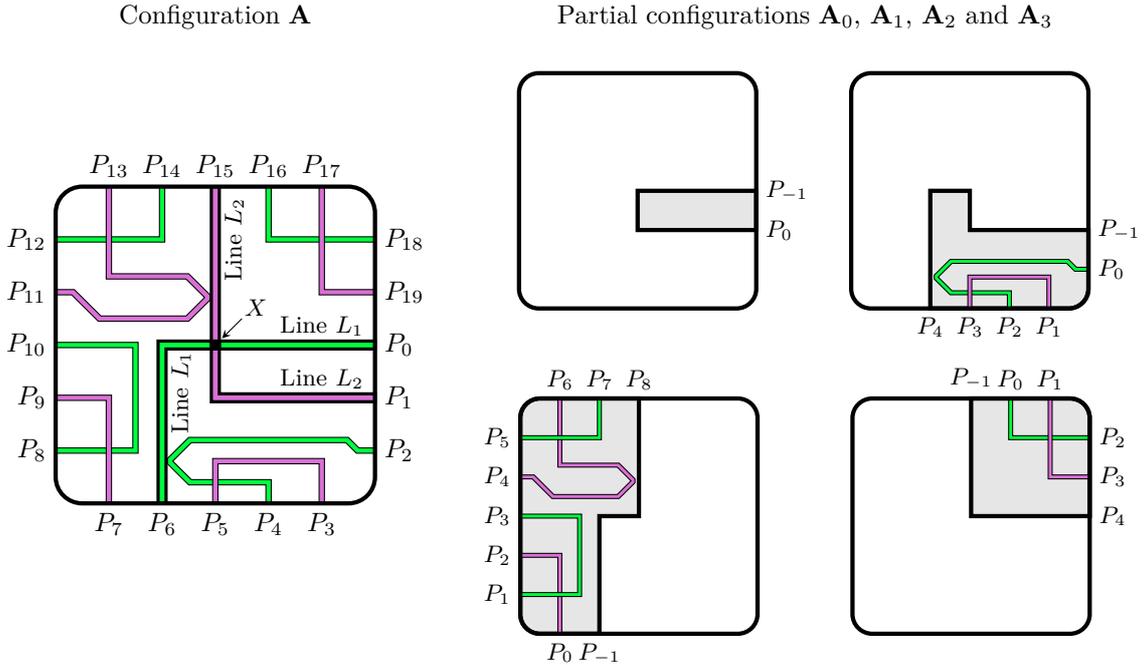
\begin{figure}[!ht]
\begin{center}
\begin{tabular}{c}
Configuration $\bA$\\[15mm]
\begin{tikzpicture}[scale=0.35,>=stealth]
\draw[fill=erin] (6,-3.9) -- (5.4,-3.9) -- (5,-3.5) -- (-1,-3.5) -- (-1.8,-4.3) -- (-1.8,-4.5) -- (-1,-5.3) -- (1.9,-5.3) -- (1.9,-6) --
(2.1,-6) -- (2.1,-5.1) -- (-0.9,-5.1) -- (-1.6,-4.4) -- (-0.9,-3.7) -- (4.9,-3.7) -- (5.3,-4.1) -- (6,-4.1) -- cycle;
\draw[fill=erin] (-6,-4.1) -- (-2.9,-4.1) -- (-2.9,0.1) -- (-6,0.1) -- (-6,-0.1) -- (-3.1,-0.1) -- (-3.1,-3.9) -- (-6,-3.9) -- cycle;
\draw[fill=erin] (-6,3.9) -- (-1.9,3.9) -- (-1.9,6) -- (-2.1,6) -- (-2.1,4.1) -- (-6,4.1) -- cycle;
\draw[fill=erin] (6,3.9) -- (1.9,3.9) -- (1.9,6) -- (2.1,6) -- (2.1,4.1) -- (6,4.1) -- cycle;
\draw[fill=orchid] (-6,1.9) -- (-5.4,1.9) -- (-4.4,0.9) -- (-1,0.9) -- (-0.2,1.7) -- (-0.2,1.9) -- (-1,2.7) -- (-3.9,2.7) -- (-3.9,6) --
(-4.1,6) -- (-4.1,2.5) -- (-1.1,2.5) -- (-0.4,1.8) -- (-1.1,1.1) -- (-4.3,1.1) -- (-5.3,2.1) -- (-6,2.1) -- cycle;
\draw[fill=orchid] (0.1,-6) -- (0.1,-4.5) -- (3.9,-4.5) -- (3.9,-6) -- (4.1,-6) -- (4.1,-4.3) -- (-0.1,-4.3) -- (-0.1,-6) -- cycle;
\draw[fill=orchid] (-3.9,-6) -- (-3.9,-1.9) -- (-6,-1.9) -- (-6,-2.1) -- (-4.1,-2.1) -- (-4.1,-6) -- cycle;
\draw[fill=orchid] (3.9,6) -- (3.9,1.9) -- (6,1.9) -- (6,2.1) -- (4.1,2.1) -- (4.1,6) -- cycle;
\draw[fill=erin,very thick] (6,-0.15) -- (-1.85,-0.15) -- (-1.85,-6) -- (-2.15,-6) -- (-2.15,0.15) -- (6,0.15) -- cycle;
\draw[fill=orchid,very thick] (6,-1.85) -- (0.15,-1.85) -- (0.15,6) -- (-0.15,6) -- (-0.15,-2.15) -- (6,-2.15) -- cycle;
\draw[ultra thick] (-5,-6) arc (-90:-180:1) -- (-6,5) arc (180:90:1) -- (5,6) arc (90:0:1) -- (6,-5) arc (0:-90:1) -- cycle;
\draw[fill=black,very thick] (-0.15,-0.15) -- (-0.15,0.15) -- (0.15,0.15) -- (0.15,-0.15) -- cycle;
\draw[->,>=stealth] (1,1) -- (0.25,0.25);
\node[anchor=south west] at (0.75,0.75) {\small $X$};
\node[anchor=south] at (4,0) {\small Line $L_1$};
\node[anchor=south] at (4,-2) {\small Line $L_2$};
\node[anchor=north,rotate=90] at (0,4) {\small Line $L_2$};
\node[anchor=north,rotate=90] at (-2,-1.8) {\small Line $L_1$};
\node[anchor=west] at (6,0) {$P_0$};
\node[anchor=west] at (6,-2) {$P_1$};
\node[anchor=west] at (6,-4) {$P_2$};
\node[anchor=north] at (4,-6) {$P_3$};
\node[anchor=north] at (2,-6) {$P_4$};
\node[anchor=north] at (0,-6) {$P_5$};
\node[anchor=north] at (-2,-6) {$P_6$};
\node[anchor=north] at (-4,-6) {$P_7$};
\node[anchor=east] at (-6,-4) {$P_8$};
\node[anchor=east] at (-6,-2) {$P_9$};
\node[anchor=east] at (-6,0) {$P_{10}$};
\node[anchor=east] at (-6,2) {$P_{11}$};
\node[anchor=east] at (-6,4) {$P_{12}$};
\node[anchor=south] at (-4,6) {$P_{13}$};
\node[anchor=south] at (-2,6) {$P_{14}$};
\node[anchor=south] at (0,6) {$P_{15}$};
\node[anchor=south] at (2,6) {$P_{16}$};
\node[anchor=south] at (4,6) {$P_{17}$};
\node[anchor=west] at (6,4) {$P_{18}$};
\node[anchor=west] at (6,2) {$P_{19}$};
\end{tikzpicture}
\\[17.2mm]
\end{tabular}
\!\!\!\!\!
\begin{tabular}{cc}
\multicolumn{2}{c}{Partial configurations $\bA_0$, $\bA_1$, $\bA_2$ and $\bA_3$}\\
\begin{tikzpicture}[scale=0.26]
\draw[ultra thick,fill=black!10] (6,0) -- (0,0) -- (0,-2) -- (6,-2) -- cycle;
\draw[ultra thick] (-5,-6) arc (-90:-180:1) -- (-6,5) arc (180:90:1) -- (5,6) arc (90:0:1) -- (6,-5) arc (0:-90:1) -- cycle;
\node[anchor=north] at (-4,-6) {\color{white}$P_7$};
\node[anchor=east] at (-6,4) {\color{white}$P_{12}$};
\node[anchor=south] at (-4,6) {\color{white}$P_{13}$};
\node[anchor=west] at (6,4) {\color{white}$P_{18}$};
\node[anchor=west] at (6,0) {\small$P_{-1}$};
\node[anchor=west] at (6,-2) {\small$P_0$};
\end{tikzpicture}
&
\begin{tikzpicture}[scale=0.26]
\draw[ultra thick,fill=black!10] (6,-2) -- (0,-2) -- (0,0) -- (-2,0) -- (-2,-6) -- (5,-6) arc (-90:0:1) -- cycle;
\draw[fill=erin] (6,-3.9) -- (5.4,-3.9) -- (5,-3.5) -- (-1,-3.5) -- (-1.8,-4.3) -- (-1.8,-4.5) -- (-1,-5.3) -- (1.9,-5.3) -- (1.9,-6) --
(2.1,-6) -- (2.1,-5.1) -- (-0.9,-5.1) -- (-1.6,-4.4) -- (-0.9,-3.7) -- (4.9,-3.7) -- (5.3,-4.1) -- (6,-4.1) -- cycle;
\draw[fill=orchid] (0.1,-6) -- (0.1,-4.5) -- (3.9,-4.5) -- (3.9,-6) -- (4.1,-6) -- (4.1,-4.3) -- (-0.1,-4.3) -- (-0.1,-6) -- cycle;
\draw[ultra thick] (-5,-6) arc (-90:-180:1) -- (-6,5) arc (180:90:1) -- (5,6) arc (90:0:1) -- (6,-5) arc (0:-90:1) -- cycle;
\node[anchor=north] at (-4,-6) {\color{white}$P_7$};
\node[anchor=south] at (-4,6) {\color{white}$P_{13}$};
\node[anchor=west] at (6,4) {\color{white}$P_{18}$};
\node[anchor=west] at (6,-2) {\small$P_{-1}$};
\node[anchor=west] at (6,-4) {\small$P_0$};
\node[anchor=north] at (4,-6) {\small$P_1$};
\node[anchor=north] at (2,-6) {\small$P_2$};
\node[anchor=north] at (0,-6) {\small$P_3$};
\node[anchor=north] at (-2,-6) {\small$P_4$};
\end{tikzpicture}
\\
\begin{tikzpicture}[scale=0.26]
\draw[ultra thick,fill=black!10] (-2,-6) -- (-2,0) -- (0,0) -- (0,6) -- (-5,6) arc (90:180:1) -- (-6,-5) arc (-180:-90:1) -- cycle;
\draw[fill=orchid] (-6,1.9) -- (-5.4,1.9) -- (-4.4,0.9) -- (-1,0.9) -- (-0.2,1.7) -- (-0.2,1.9) -- (-1,2.7) -- (-3.9,2.7) -- (-3.9,6) --
(-4.1,6) -- (-4.1,2.5) -- (-1.1,2.5) -- (-0.4,1.8) -- (-1.1,1.1) -- (-4.3,1.1) -- (-5.3,2.1) -- (-6,2.1) -- cycle;
\draw[fill=orchid] (-3.9,-6) -- (-3.9,-1.9) -- (-6,-1.9) -- (-6,-2.1) -- (-4.1,-2.1) -- (-4.1,-6) -- cycle;
\draw[fill=erin] (-6,-4.1) -- (-2.9,-4.1) -- (-2.9,0.1) -- (-6,0.1) -- (-6,-0.1) -- (-3.1,-0.1) -- (-3.1,-3.9) -- (-6,-3.9) -- cycle;
\draw[fill=erin] (-6,3.9) -- (-1.9,3.9) -- (-1.9,6) -- (-2.1,6) -- (-2.1,4.1) -- (-6,4.1) -- cycle;
\draw[ultra thick] (-5,-6) arc (-90:-180:1) -- (-6,5) arc (180:90:1) -- (5,6) arc (90:0:1) -- (6,-5) arc (0:-90:1) -- cycle;
\node[anchor=north] at (-4,-6) {\color{white}$P_7$};
\node[anchor=east] at (-6,4) {\color{white}$P_{12}$};
\node[anchor=south] at (-4,6) {\color{white}$P_{13}$};
\node[anchor=west] at (6,4) {\color{white}$P_{18}$};
\node[anchor=north] at (-2,-6) {\small$P_{-1}$};
\node[anchor=north] at (-4,-6) {\small$P_0$};
\node[anchor=east] at (-6,-4) {\small$P_1$};
\node[anchor=east] at (-6,-2) {\small$P_2$};
\node[anchor=east] at (-6,0) {\small$P_3$};
\node[anchor=east] at (-6,2) {\small$P_4$};
\node[anchor=east] at (-6,4) {\small$P_5$};
\node[anchor=south] at (-4,6) {\small$P_6$};
\node[anchor=south] at (-2,6) {\small$P_7$};
\node[anchor=south] at (0,6) {\small$P_8$};
\end{tikzpicture}
&
\begin{tikzpicture}[scale=0.26]
\draw[ultra thick,fill=black!10] (6,0) -- (0,0) -- (0,6) -- (5,6) arc (90:0:1) -- cycle;
\draw[fill=erin] (6,3.9) -- (1.9,3.9) -- (1.9,6) -- (2.1,6) -- (2.1,4.1) -- (6,4.1) -- cycle;
\draw[fill=orchid] (3.9,6) -- (3.9,1.9) -- (6,1.9) -- (6,2.1) -- (4.1,2.1) -- (4.1,6) -- cycle;
\draw[ultra thick] (-5,-6) arc (-90:-180:1) -- (-6,5) arc (180:90:1) -- (5,6) arc (90:0:1) -- (6,-5) arc (0:-90:1) -- cycle;
\node[anchor=north] at (-4,-6) {\color{white}$P_7$};
\node[anchor=south] at (-4,6) {\color{white}$P_{13}$};
\node[anchor=west] at (6,4) {\color{white}$P_{18}$};
\node[anchor=west] at (6,0) {\small$P_4$};
\node[anchor=south] at (0,6) {\small$P_{-1}$};
\node[anchor=south] at (2,6) {\small$P_0$};
\node[anchor=south] at (4,6) {\small$P_1$};
\node[anchor=west] at (6,4) {\small$P_2$};
\node[anchor=west] at (6,2) {\small$P_3$};
\end{tikzpicture}
\end{tabular}
\caption{Splitting a configuration into a collection of partial configurations}
\end{center}
\label{fig:2}
\end{figure}
\begin{lemma}
\label{lem:16}
We have $\bN_2 = \bN_1 \bN_3$.
\end{lemma}
\begin{proof}
We design a bijection $\varphi : N_2 \mapsto N_3 \times N_1$,
such that $\varphi(\bA) = (\bA_0,\bA_1)$ where
$\deg \bA = \deg \bA_0 + \deg \bA_1$ and
$\codim{\bA} = \codim {\bA_0} + \codim {\bA_1}$ for all partial configurations $\bA$.

Let $\bA$ be a partial configuration of degree $d$ and let $\bC$ be the connected component of $\bA$
that contains the point $P_{-1}$. Let $a$ be the greatest element of $\{-1,\ldots,4d-1\}$
such that $P_a \in \bC$. Furthermore, let $\mathcal{D}$ be the unique connected component of $\bU \setminus \bC$
whose border contains the arc of circle $[P_a,P_{4d}]$.
First, deleting the points $P_x$ that do not belong to the arc of circle $[a+1,4d-1]$
and renaming every point $P_x$ (with $a < x < 4d$ to $P_{x-a-1}$,
we observe that $\bA \cap \mathcal{D}$ is a configuration (whence $a \equiv 3 \pmod{4}$).

Second, deleting the points $P_x$ with $a < x < 4d$ and renaming the point $P_{4d}$ to $P_{a+1}$,
we also observe that $\bA \setminus \mathcal{D}$ is a widespread partial configuration.
Hence, we define $\varphi(\bA)$ as the pair $(\bA \setminus \mathcal{D},\bA \cap \mathcal{D})$,
as illustrated in Fig.~\ref{fig:3}.
By construction, every crossing point or meeting point of $\bA$ belongs to either
$\bA \setminus \mathcal{D}$ or $\bA \cap \mathcal{D}$, and $\varphi$ is clearly bijective.
Lemma~\ref{lem:16} follows.
\end{proof}
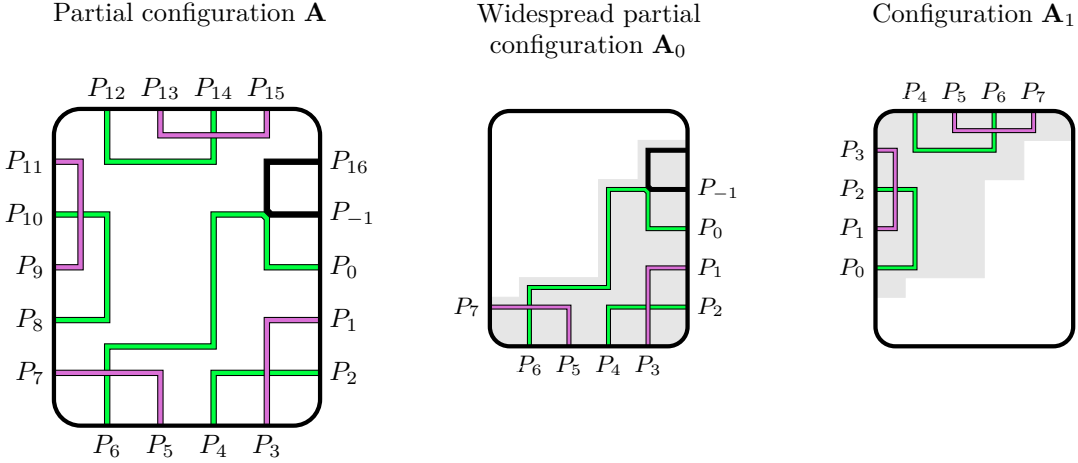
\begin{figure}[!ht]
\begin{center}
\begin{tabular}{c}
Partial configuration $\bA$\\[5mm]
\begin{tikzpicture}[scale=0.35,>=stealth]
\draw[fill=black] (5,1.9) -- (3.1,1.9) -- (2.9,2.1) -- (2.9,4.1) -- (5,4.1) -- (5,3.9) -- (3.1,3.9) -- (3.1,2.2) -- (3.2,2.1) -- (5,2.1) -- cycle;
\draw[fill=erin] (5,0.1) -- (3.1,0.1) -- (3.1,1.9) -- (2.9,2.1) -- (0.9,2.1) -- (0.9,-2.9) -- (-3.1,-2.9) --
(-3.1,-6) -- (-2.9,-6) -- (-2.9,-3.1) -- (1.1,-3.1) -- (1.1,1.9) -- (2.8,1.9) -- (2.9,1.8) -- (2.9,-0.1) -- (5,-0.1) -- cycle;
\draw[fill=erin] (5,-3.9) -- (0.9,-3.9) -- (0.9,-6) -- (1.1,-6) -- (1.1,-4.1) -- (5,-4.1) -- cycle;
\draw[fill=erin] (-3.1,6) -- (-3.1,3.9) -- (1.1,3.9) -- (1.1,6) -- (0.9,6) -- (0.9,4.1) -- (-1.1,4.1) -- (-2.9,4.1) -- (-2.9,6) -- cycle;
\draw[fill=erin] (-5,2.1) -- (-2.9,2.1) -- (-2.9,-2.1) -- (-5,-2.1) -- (-5,-1.9) -- (-3.1,-1.9) -- (-3.1,1.9) -- (-5,1.9) -- cycle;
\draw[fill=orchid] (5,-1.9) -- (2.9,-1.9) -- (2.9,-6) -- (3.1,-6) -- (3.1,-2.1) -- (5,-2.1) -- cycle;
\draw[fill=orchid] (-0.9,-6) -- (-0.9,-3.9) -- (-5,-3.9) -- (-5,-4.1) -- (-1.1,-4.1) -- (-1.1,-6) -- cycle;
\draw[fill=orchid] (-5,-0.1) -- (-3.9,-0.1) -- (-3.9,4.1) -- (-5,4.1) -- (-5,3.9) -- (-4.1,3.9) -- (-4.1,0.1) -- (-5,0.1) -- cycle;
\draw[fill=orchid] (3.1,6) -- (3.1,4.9) -- (-1.1,4.9) -- (-1.1,6) -- (-0.9,6) -- (-0.9,5.1) -- (2.9,5.1) -- (2.9,6) -- cycle;
\draw[ultra thick] (-4,-6) arc (-90:-180:1) -- (-5,5) arc (180:90:1) -- (4,6) arc (90:0:1) -- (5,-5) arc (0:-90:1) -- cycle;
\node[anchor=west] at (5,0) {$P_0$};
\node[anchor=west] at (5,-2) {$P_1$};
\node[anchor=west] at (5,-4) {$P_2$};
\node[anchor=north] at (3,-6) {$P_3$};
\node[anchor=north] at (1,-6) {$P_4$};
\node[anchor=north] at (-1,-6) {$P_5$};
\node[anchor=north] at (-3,-6) {$P_6$};
\node[anchor=east] at (-5,-4) {$P_7$};
\node[anchor=east] at (-5,-2) {$P_8$};
\node[anchor=east] at (-5,0) {$P_9$};
\node[anchor=east] at (-5,2) {$P_{10}$};
\node[anchor=east] at (-5,4) {$P_{11}$};
\node[anchor=south] at (-3,6) {$P_{12}$};
\node[anchor=south] at (-1,6) {$P_{13}$};
\node[anchor=south] at (1,6) {$P_{14}$};
\node[anchor=south] at (3,6) {$P_{15}$};
\node[anchor=west] at (5,4) {$P_{16}$};
\node[anchor=west] at (5,2) {$P_{-1}$};
\end{tikzpicture}
\end{tabular}
\!\!\!\!\!
\begin{tabular}{cc}
Widespread partial & Configuration $\bA_1$ \\
configuration $\bA_0$ & \\[1mm]
\begin{tikzpicture}[scale=0.26]
\draw[draw=black!10,fill=black!10] (5,4.5) -- (2.5,4.5) -- (2.5,2.5) -- (0.5,2.5) -- (0.5,-2.5) -- (-3.5,-2.5) --
(-3.5,-3.5) -- (-5,-3.5) -- (-5,-5) arc (-180:-90:1) -- (4,-6) arc (-90:0:1) -- cycle;
\draw[fill=black] (5,1.9) -- (3.1,1.9) -- (2.9,2.1) -- (2.9,4.1) -- (5,4.1) -- (5,3.9) -- (3.1,3.9) -- (3.1,2.2) -- (3.2,2.1) -- (5,2.1) -- cycle;
\draw[fill=erin] (5,0.1) -- (3.1,0.1) -- (3.1,1.9) -- (2.9,2.1) -- (0.9,2.1) -- (0.9,-2.9) -- (-3.1,-2.9) --
(-3.1,-6) -- (-2.9,-6) -- (-2.9,-3.1) -- (1.1,-3.1) -- (1.1,1.9) -- (2.8,1.9) -- (2.9,1.8) -- (2.9,-0.1) -- (5,-0.1) -- cycle;
\draw[fill=erin] (5,-3.9) -- (0.9,-3.9) -- (0.9,-6) -- (1.1,-6) -- (1.1,-4.1) -- (5,-4.1) -- cycle;
\draw[fill=orchid] (5,-1.9) -- (2.9,-1.9) -- (2.9,-6) -- (3.1,-6) -- (3.1,-2.1) -- (5,-2.1) -- cycle;
\draw[fill=orchid] (-0.9,-6) -- (-0.9,-3.9) -- (-5,-3.9) -- (-5,-4.1) -- (-1.1,-4.1) -- (-1.1,-6) -- cycle;
\draw[ultra thick] (-4,-6) arc (-90:-180:1) -- (-5,5) arc (180:90:1) -- (4,6) arc (90:0:1) -- (5,-5) arc (0:-90:1) -- cycle;
\node[anchor=west] at (5,0) {\small$P_0$};
\node[anchor=west] at (5,-2) {\small$P_1$};
\node[anchor=west] at (5,-4) {\small$P_2$};
\node[anchor=north] at (3,-6) {\small$P_3$};
\node[anchor=north] at (1,-6) {\small$P_4$};
\node[anchor=north] at (-1,-6) {\small$P_5$};
\node[anchor=north] at (-3,-6) {\small$P_6$};
\node[anchor=east] at (-5,-4) {\small$P_7$};
\node[anchor=west] at (5,2) {\small$P_{-1}$};
\node[anchor=north] at (-4,-6) {\color{white}$P_7$};
\node[anchor=east] at (-6,4) {\color{white}$P_{12}$};
\node[anchor=south] at (-4,6) {\color{white}$P_{13}$};
\node[anchor=west] at (6,4) {\color{white}$P_{18}$};
\end{tikzpicture}
&
\begin{tikzpicture}[scale=0.26]
\draw[draw=black!10,fill=black!10] (5,4.5) -- (2.5,4.5) -- (2.5,2.5) -- (0.5,2.5) -- (0.5,-2.5) -- (-3.5,-2.5) --
(-3.5,-3.5) -- (-5,-3.5) -- (-5,5) arc (180:90:1) -- (4,6) arc (90:0:1) -- cycle;
\draw[fill=erin] (-3.1,6) -- (-3.1,3.9) -- (1.1,3.9) -- (1.1,6) -- (0.9,6) -- (0.9,4.1) -- (-1.1,4.1) -- (-2.9,4.1) -- (-2.9,6) -- cycle;
\draw[fill=erin] (-5,2.1) -- (-2.9,2.1) -- (-2.9,-2.1) -- (-5,-2.1) -- (-5,-1.9) -- (-3.1,-1.9) -- (-3.1,1.9) -- (-5,1.9) -- cycle;
\draw[fill=orchid] (-5,-0.1) -- (-3.9,-0.1) -- (-3.9,4.1) -- (-5,4.1) -- (-5,3.9) -- (-4.1,3.9) -- (-4.1,0.1) -- (-5,0.1) -- cycle;
\draw[fill=orchid] (3.1,6) -- (3.1,4.9) -- (-1.1,4.9) -- (-1.1,6) -- (-0.9,6) -- (-0.9,5.1) -- (2.9,5.1) -- (2.9,6) -- cycle;
\draw[ultra thick] (-4,-6) arc (-90:-180:1) -- (-5,5) arc (180:90:1) -- (4,6) arc (90:0:1) -- (5,-5) arc (0:-90:1) -- cycle;
\node[anchor=east] at (-5,-2) {\small$P_0$};
\node[anchor=east] at (-5,0) {\small$P_1$};
\node[anchor=east] at (-5,2) {\small$P_2$};
\node[anchor=east] at (-5,4) {\small$P_3$};
\node[anchor=south] at (-3,6) {\small$P_4$};
\node[anchor=south] at (-1,6) {\small$P_5$};
\node[anchor=south] at (1,6) {\small$P_6$};
\node[anchor=south] at (3,6) {\small$P_7$};
\node[anchor=north] at (-4,-6) {\color{white}$P_7$};
\node[anchor=east] at (-6,4) {\color{white}$P_{12}$};
\node[anchor=south] at (-4,6) {\color{white}$P_{13}$};
\node[anchor=west] at (6,4) {\color{white}$P_{18}$};
\end{tikzpicture}
\\[10.5mm]
\end{tabular}
\caption{Splitting a partial configuration into a widespread partial configuration and a configuration}
\end{center}
\label{fig:3}
\end{figure}
\begin{lemma}
\label{lem:17}
We have $\bN_2 = \bN_1 + \sum_{k \geq 1} x^{2k-1} y^k \bN_2^{4k+1} \bN_3^k$.
\end{lemma}
\begin{proof}
Before defining suitable bijections, we first
partition the set $N_2$ as follows.
Let $\bA$ be a partial configuration and let $L$ be the base line of $\bA$.
If $L$ contains no meeting point, then we say that $\bA$ has type $0$.
Otherwise, let $X$ be the first meeting point of $L$ (while going from $P_{-1}$ to $P_{4d}$):
if $k \geq 1$ even lines of $\bA$ meet $L$ at that point,
then we say that $\bA$ has type $k$.
Now, for all $k \geq 0$, we denote by $N_{2,k}$ the set of
partial configurations of type $k$.

We first design a bijection $\varphi_0 : N_{2,0} \mapsto N_1$
as follows: if $\bA$ has type $0$ and degree $d$, then $\varphi(\bA)$ is
the configuration obtained by deleting the points $P_{-1}$, $P_{4d}$ and the base line of $\bA$.
The mapping $\varphi_1$ is clearly bijective, and leaves both the degree and the codimension unchanged.

Then, for all $k \geq 1$, we design a bijection $\varphi_k : N_{2,k} \mapsto N_2 \times (N_2^4 \times N_3)^k$
such that $\varphi(\bA) = (\bA_i)_{0 \leq i \leq 5k}$, where
\begin{align}
\label{align:18}
\deg \bA = \sum_{i=0}^{5k} \deg \bA_i + k \text{ and } \codim{\bA} = \sum_{i=0}^{5k} \codim{\bA_i} + k + \mathbf{1}_{k \geq 2}.
\end{align}
Let $\bA$ be a partial configuration of type $k$ and degree $d$, let $\bS$ be the canonical splitting of $\bA$;
hereafter we consider exclusively lines in $\bS$.
Let $L^e_1,\ldots,L^e_k$ be the even lines that touch $L$ at the meeting point $X$;
let $L^o_1,\ldots,L^o_k$ be the odd lines that cross respectively $L^e_1,\ldots,L^e_k$.
For $i \in \{1,\ldots,k\}$, let $P_{a_{0,i}} < P_{a_{2,i}}$ be the endpoints of $L^e_i$ and
let $P_{a_{1,i}} < P_{a_{3,i}}$ be the endpoints of $L^o_i$.
In addition, let $x$ be the greatest integer such that $P_{a_{3,i}}$ and $P_{a_{4,i}}$
belong to the same connected component of $\bA \setminus \{X\}$:
we set $a_{4,i} = x+1/2$ and we add a new point $P_x$ on the open arc of circle $(P_{a_{4,i}},P_{a_{4,i}+1})$.
Assuming that $a_{0,1} < a_{0,2} < \ldots < a_{0,k}$, one checks easily that
\[0 \leq a_{0,1} < a_{1,1} < a_{2,1} < a_{3,1} < a_{4,1} < a_{0,2} < \ldots < a_{4,k} < 4d.\]
Then, let $L^z_i$ be a (new) piecewise-affine line with endpoints $X$ and $P_{a_{4,i}}$ and that lies in $\bU \setminus \bA$.
The set $L \cup \bigcup_{i=1}^k (L^e_i \cup L^o_i \cup L^z_i)$ splits the unit disk $\bU$ into $5k+2$ connected components:
\begin{itemize}
 \item for all $0 \leq u \leq 4$ and $1 \leq i \leq k$, one component $\bC_{u,i}$ whose border contains the arc of circle $[P_{a_{u,i}},P_{a_{u+1,i}}]$
 (with the convention that $a_{5,i} = a_{0,i+1}$ when $i < k$, and $a_{5,k} = 4d$);
 \item one component $\bC_0$ whose border contains the arc of circle $[P_{-1},P_{a_{0,1}}]$;
 \item one component $\bC_{-1}$ whose border contains the arc of circle $[P_{4d},P_{-1}]$.
\end{itemize}

Let $\bC \neq \bC_{-1}$ be one such component.
Up to deleting the points $P_x$ (where $x$ is an integer or a half-integer of the form $a_{4,i}$)
that do not belong to $\partial \bC$ and to renumbering the other points $P_x$
from $P_{-1}$ to $P_\ell$ (where there are $\ell+2$ points $P_x$ on $\partial \bC$),
we observe that $(\bA \cup \bigcup_{i=1}^k L^z_i) \cap \overline{\bC}$ is a partial configuration;
and is even a widespread partial configuration if $\bC = \bC_{3,i}$ for some $i$.
We denote below this partial configuration by $\bA_{u,i}$ if $\bC = \bC_{u,i}$, or by
$\bA_0$ if $\bC = \bC_0$.

Hence, we set $\varphi_k(\bA) = (\bA_0,\bA_{0,1},\bA_{1,1},\bA_{2,1},\bA_{3,1},\bA_{4,1},\bA_{0,2},\ldots,\bA_{4,k})$,
as illustrated in Fig.~\ref{fig:4} (in the case $k = 2$).
It is easy to check that $\varphi_k$ is bijective. Moreover,
every crossing point of $\bA$ besides those between lines $L^e_i$ and $L^o_i$ belongs to some $\bA_{u,i}$ (or to $\bA_0$)
and every meeting point of $\bA$ besides $X$ belongs to some $\bA_{u,i}$ (or to $\bA_0$) as well,
with the same codimension. This proves that $\varphi_k$ satisfies~(\ref{align:18}).
Lemma~\ref{lem:17} follows.
\end{proof}
\begin{figure}[!ht]
\begin{center}
{\setlength{\tabcolsep}{-3mm}
\begin{tabular}{cccc}
\multicolumn{2}{c}{Partial configuration $\bA$} & \multicolumn{2}{c}{Associated partial configurations} \\[2mm]
\multicolumn{2}{c}{\raisebox{-43mm}{
\begin{tikzpicture}[scale=0.35]
\draw[fill=black] (8,-1.1) -- (5.9,-1.1) -- (5.9,0.9) -- (6.1,1.1) -- (8,1.1) -- (8,0.9) -- (6.2,0.9) -- (6.1,0.8) -- (6.1,-0.9) -- (8,-0.9) -- cycle;
\draw[thick,densely dotted] (-3,-9) -- (-3,-4) -- (3,-4) -- (5.7,-1.3);
\draw[thick,densely dotted] (-3,9) -- (-3,6) -- (2,6) -- (2,0.5) -- (5.4,0.5) -- (5.9,-0.9);
\draw[<-,>=stealth,thick] (-3,-9.1) -- (-3,-10.5);
\draw[<-,>=stealth,thick] (-3,9.1) -- (-3,10.5);
\draw[fill=erin] (8,-2.9) -- (7.1,-2.9) -- (7.1,-2.1) -- (6.1,-1.1) -- (5.9,-1.1) -- (5.7,-1.3) -- (5.7,-6) -- (6.8,-7.1) -- (8,-7.1) --
(8,-6.9) -- (6.9,-6.9) -- (5.9,-5.9) -- (5.9,-1.4) -- (6,-1.3) -- (6.9,-2.2) -- (6.9,-3.1) -- (8,-3.1) -- cycle;
\draw[fill=erin] (4.1,-9) -- (4.1,-7.9) -- (-0.1,-7.9) -- (-0.1,-9) -- (0.1,-9) -- (0.1,-8.1) -- (3.9,-8.1) -- (3.9,-9) -- cycle;
\draw[fill=erin] (-3.9,-9) -- (-3.9,-6.9) -- (-8,-6.9) -- (-8,-7.1) -- (-4.1,-7.1) -- (-4.1,-9) -- cycle;
\draw[fill=erin] (-2.1,9) -- (-2.1,6.9) -- (2.1,6.9) -- (2.1,9) -- (1.9,9) -- (1.9,7.1) -- (-1.9,7.1) -- (-1.9,9) -- cycle;
\draw[fill=erin] (-8,-3.1) -- (0.1,-3.1) -- (0.1,-1.3) -- (5.7,-1.3) -- (5.9,-1.1) -- (5.9,-0.9) -- (4.9,0.1) -- (0.1,0.1) -- (0.1,5.1) --
(-5.9,5.1) -- (-6.1,4.9) -- (-6.1,1.1) -- (-8,1.1) -- (-8,0.9) -- (-5.9,0.9) -- (-5.9,4.8) -- (-5.8,4.9) -- (-0.1,4.9) -- (-0.1,-0.1) --
(4.8,-0.1) -- (5.7,-1) -- (5.6,-1.1) -- (-0.1,-1.1) -- (-0.1,-2.9) -- (-8,-2.9) -- cycle;
\draw[fill=erin] (-8,4.9) -- (-6.1,4.9) -- (-5.9,5.1) -- (-5.9,9) -- (-6.1,9) -- (-6.1,5.2) -- (-6.2,5.1) -- (-8,5.1) -- cycle;
\draw[fill=erin] (8,4.9) -- (6.1,4.9) -- (6.1,1.1) -- (5.9,0.9) -- (3.9,0.9) -- (3.9,7.1) -- (5.9,7.1) -- (5.9,9) -- (6.1,9) --
(6.1,6.9) -- (4.1,6.9) -- (4.1,1.1) -- (5.8,1.1) -- (5.9,1.2) -- (5.9,5.1) -- (8,5.1) -- cycle;
\draw[fill=orchid] (6.1,-9) -- (6.1,-7.8) -- (5.2,-6.9) -- (2.1,-6.9) -- (1.9,-7.1) -- (1.9,-9) -- (2.1,-9) -- (2.1,-7.2) -- (2.2,-7.1) --
(5.1,-7.1) -- (5.9,-7.9) -- (5.9,-9) -- cycle;
\draw[fill=orchid] (8,-5.1) -- (2.1,-5.1) -- (2.1,-6.9) -- (1.9,-7.1) -- (-1.9,-7.1) -- (-1.9,-9) -- (-2.1,-9) -- (-2.1,-6.9) --
(1.8,-6.9) -- (1.9,-6.8) -- (1.9,-4.9) -- (8,-4.9) -- cycle;
\draw[fill=orchid] (-5.9,-9) -- (-5.9,-4.9) -- (-8,-4.9) -- (-8,-5.1) -- (-6.1,-5.1) -- (-6.1,-9) -- cycle;
\draw[fill=orchid] (-8,-1.1) -- (-6.9,-1.1) -- (-6.9,3.1) -- (-8,3.1) -- (-8,2.9) -- (-7.1,2.9) -- (-7.1,-0.9) -- (-8,-0.9) -- cycle;
\draw[fill=orchid] (-3.9,9) -- (-3.9,6.9) -- (-8,6.9) -- (-8,7.1) -- (-4.1,7.1) -- (-4.1,9) -- cycle;
\draw[fill=orchid] (-0.1,9) -- (-0.1,7.9) -- (4.1,7.9) -- (4.1,9) -- (3.9,9) -- (3.9,8.1) -- (0.1,8.1) -- (0.1,9) -- cycle;
\draw[fill=orchid] (8,7.1) -- (6.9,7.1) -- (6.9,2.9) -- (8,2.9) -- (8,3.1) -- (7.1,3.1) -- (7.1,6.9) -- (8,6.9) -- cycle;
\draw[ultra thick] (-7,-9) arc (-90:-180:1) -- (-8,8) arc (180:90:1) -- (7,9) arc (90:0:1) -- (8,-8) arc (0:-90:1) -- cycle;
\node[anchor=west] at (8,-3) {$P_0$};
\node[anchor=west] at (8,-5) {$P_1$};
\node[anchor=west] at (8,-7) {$P_2$};
\node[anchor=north] at (6,-9) {$P_3$};
\node[anchor=north] at (4,-9) {$P_4$};
\node[anchor=north] at (2,-9) {$P_5$};
\node[anchor=north] at (0,-9) {$P_6$};
\node[anchor=north] at (-2,-9) {$P_7$};
\node[anchor=north] at (-4,-9) {$P_8$};
\node[anchor=north] at (-6,-9) {$P_9$};
\node[anchor=east] at (-8,-7) {$P_{10}$};
\node[anchor=east] at (-8,-5) {$P_{11}$};
\node[anchor=east] at (-8,-3) {$P_{12}$};
\node[anchor=east] at (-8,-1) {$P_{13}$};
\node[anchor=east] at (-8,1) {$P_{14}$};
\node[anchor=east] at (-8,3) {$P_{15}$};
\node[anchor=east] at (-8,5) {$P_{16}$};
\node[anchor=east] at (-8,7) {$P_{17}$};
\node[anchor=south] at (-6,9.05) {$P_{18}$};
\node[anchor=south] at (-4,9.05) {$P_{19}$};
\node[anchor=south] at (-2,9.05) {$P_{20}$};
\node[anchor=south] at (0,9.05) {$P_{21}$};
\node[anchor=south] at (2,9.05) {$P_{22}$};
\node[anchor=south] at (4,9.05) {$P_{23}$};
\node[anchor=south] at (6,9.05) {$P_{24}$};
\node[anchor=west] at (8,7) {$P_{25}$};
\node[anchor=west] at (8,5) {$P_{26}$};
\node[anchor=west] at (8,3) {$P_{27}$};
\node[anchor=west] at (8,1) {$P_{28}$};
\node[anchor=west] at (8,-1) {$P_{-1}$};
\node[anchor=north] at (-3,-10.5) {$P_{7.5}$};
\node[anchor=south] at (-3,10.5) {$P_{19.5}$};
\end{tikzpicture}
}\vspace*{-42mm}}
&
\begin{tikzpicture}[scale=0.18]
\node[anchor=south] at (-2,9) {\color{white}\small$P_4$};
\node[anchor=east] at (-7.7,-1) {\color{white}\small$P_{-1}$};
\node[anchor=west] at (7.7,-1) {\color{white}\small$P_{-1}$};
\node[anchor=north] at (-2,-9) {\color{white}\small$P_4$};
\node[anchor=north] at (0,8.5) {$\bA_0$};
\draw[draw=black!10,fill=black!10] (8,-3) -- (7,-3) -- (7,-2.1) -- (6,-1.1) -- (8,-1.1) -- cycle;
\draw[fill=black] (8,-2.9) -- (7.1,-2.9) -- (7.1,-2.1) -- (6.1,-1.1) -- (8,-1.1) -- (8,-0.9) -- (5.9,-0.9) -- (5.9,-1.2) --
(6.9,-2.2) -- (6.9,-3.1) -- (8,-3.1) -- cycle;
\draw[ultra thick] (-7,-9) arc (-90:-180:1) -- (-8,8) arc (180:90:1) -- (7,9) arc (90:0:1) -- (8,-8) arc (0:-90:1) -- cycle;
\node[anchor=west] at (7.7,-3) {\small$P_0$};
\node[anchor=west] at (7.7,-1) {\small$P_{-1}$};
\end{tikzpicture}
&
\begin{tikzpicture}[scale=0.18]
\node[anchor=south] at (-2,9) {\color{white}\small$P_4$};
\node[anchor=east] at (-7.7,-5) {\color{white}\small$P_{-1}$};
\node[anchor=west] at (7.7,-1) {\color{white}\small$P_{-1}$};
\node[anchor=north] at (-2,-9) {\color{white}\small$P_4$};
\node[anchor=north] at (0,8.5) {$\bA_{0,1}$};
\draw[draw=black!10,fill=black!10] (8,-4.9) -- (5.9,-4.9) -- (5.9,-1.4) -- (6,-1.3) -- (6.9,-2.2) -- (6.9,-3.1) -- (8,-3.1) -- cycle;
\draw[fill=black] (8,-2.9) -- (7.1,-2.9) -- (7.1,-2.1) -- (6.1,-1.1) -- (5.9,-1.1) -- (5.7,-1.3) -- (5.7,-5.1) -- (8,-5.1) -- (8,-4.9) --
(5.9,-4.9) -- (5.9,-1.4) -- (6,-1.3) -- (6.9,-2.2) -- (6.9,-3.1) -- (8,-3.1) -- cycle;
\draw[ultra thick] (-7,-9) arc (-90:-180:1) -- (-8,8) arc (180:90:1) -- (7,9) arc (90:0:1) -- (8,-8) arc (0:-90:1) -- cycle;
\node[anchor=west] at (7.7,-3) {\small$P_{-1}$};
\node[anchor=west] at (7.7,-5) {\small$P_0$};
\end{tikzpicture}
\\
&&\begin{tikzpicture}[scale=0.18]
\node[anchor=south] at (-2,9) {\color{white}\small$P_4$};
\node[anchor=east] at (-7.7,-1) {\color{white}\small$P_{-1}$};
\node[anchor=west] at (7.7,-1) {\color{white}\small$P_{-1}$};
\node[anchor=north] at (-2,-9) {\color{white}\small$P_4$};
\node[anchor=north] at (0,8.5) {$\bA_{1,1}$};
\draw[draw=black!10,fill=black!10] (8,-6.9) -- (6.9,-6.9) -- (5.9,-5.9) -- (5.9,-5.1) -- (8,-5.1) -- cycle;
\draw[fill=black]  (8,-4.9) -- (5.7,-4.9) -- (5.7,-6) -- (6.8,-7.1) -- (8,-7.1) --
(8,-6.9) -- (6.9,-6.9) -- (5.9,-5.9) -- (5.9,-5.1) -- (8,-5.1) -- cycle;
\draw[ultra thick] (-7,-9) arc (-90:-180:1) -- (-8,8) arc (180:90:1) -- (7,9) arc (90:0:1) -- (8,-8) arc (0:-90:1) -- cycle;
\node[anchor=west] at (7.7,-5) {\small$P_{-1}$};
\node[anchor=west] at (7.7,-7) {\small$P_0$};
\end{tikzpicture}
&
\begin{tikzpicture}[scale=0.18]
\node[anchor=south] at (-2,9) {\color{white}\small$P_4$};
\node[anchor=east] at (-7.7,-1) {\color{white}\small$P_{-1}$};
\node[anchor=west] at (7.7,-1) {\color{white}\small$P_{-1}$};
\node[anchor=north] at (-2,-9) {\color{white}\small$P_4$};
\node[anchor=north] at (0,8.5) {$\bA_{2,1}$};
\draw[draw=black!10,fill=black!10] (8,-7.1) -- (6.8,-7.1) -- (5.7,-6) -- (5.7,-5.1) -- (2.1,-5.1) -- (2.1,-6.9) -- (1.9,-7.1) -- (-1.9,-7.1) -- (-1.9,-9) -- (7,-9) arc (-90:0:1) -- cycle;
\draw[fill=erin] (6.1,-9) -- (6.1,-7.8) -- (5.2,-6.9) -- (2.1,-6.9) -- (1.9,-7.1) -- (1.9,-9) -- (2.1,-9) -- (2.1,-7.2) -- (2.2,-7.1) --
(5.1,-7.1) -- (5.9,-7.9) -- (5.9,-9) -- cycle;
\draw[fill=orchid] (4.1,-9) -- (4.1,-7.9) -- (-0.1,-7.9) -- (-0.1,-9) -- (0.1,-9) -- (0.1,-8.1) -- (3.9,-8.1) -- (3.9,-9) -- cycle;
\draw[fill=black] (8,-7.1) -- (6.8,-7.1) -- (5.7,-6) -- (5.7,-5.1) -- (2.1,-5.1) -- (2.1,-6.9) -- (1.9,-7.1) -- (-1.9,-7.1) -- (-1.9,-9) --
(-2.1,-9) -- (-2.1,-6.9) -- (1.8,-6.9) -- (1.9,-6.8) -- (1.9,-4.9) -- (5.9,-4.9) -- (5.9,-5.9) -- (6.9,-6.9) -- (8,-6.9) -- cycle;
\draw[ultra thick] (-7,-9) arc (-90:-180:1) -- (-8,8) arc (180:90:1) -- (7,9) arc (90:0:1) -- (8,-8) arc (0:-90:1) -- cycle;
\node[anchor=west] at (7.7,-7) {\small$P_{-1}$};
\node[anchor=north] at (6,-9) {\small$P_0$};
\node[anchor=north] at (4,-9) {\small$P_1$};
\node[anchor=north] at (2,-9) {\small$P_2$};
\node[anchor=north] at (0,-9) {\small$P_3$};
\node[anchor=north] at (-2,-9) {\small$P_4$};
\end{tikzpicture}
\\
\begin{tikzpicture}[scale=0.18]
\node[anchor=south] at (-2,9) {\color{white}\small$P_4$};
\node[anchor=east] at (-7.7,-1) {\color{white}\small$P_{-1}$};
\node[anchor=west] at (7.7,-1) {\color{white}\small$P_{-1}$};
\node[anchor=north] at (-2,-9) {\color{white}\small$P_4$};
\node[anchor=north] at (0,8.5) {$\bA_{3,1}$};
\draw[draw=black!10,fill=black!10] (-2.1,-9) -- (-2.1,-6.9) -- (1.8,-6.9) -- (1.9,-6.8) -- (1.9,-4.9) -- (5.7,-4.9) -- (5.7,-1.3) -- (2.9,-4.1) -- (-2.9,-4.1) -- (-2.9,-9) -- cycle;
\draw[fill=black] (-3.1,-9) -- (-3.1,-3.9) -- (2.8,-3.9) -- (5.6,-1.1) -- (5.9,-1.1) -- (5.9,-5.1) -- (2.1,-5.1) -- (2.1,-6.9) -- (1.9,-7.1) -- (-1.9,-7.1) --
(-1.9,-9) -- (-2.1,-9) -- (-2.1,-6.9) -- (1.8,-6.9) -- (1.9,-6.8) -- (1.9,-4.9) -- (5.7,-4.9) -- (5.7,-1.3) -- (2.9,-4.1) -- (-2.9,-4.1) -- (-2.9,-9) -- cycle;
 
\draw[ultra thick] (-7,-9) arc (-90:-180:1) -- (-8,8) arc (180:90:1) -- (7,9) arc (90:0:1) -- (8,-8) arc (0:-90:1) -- cycle;
\draw[->,>=stealth,thick] (0,-10.5) -- (-0.75,-10.5) -- (-2,-9.25);
\draw[->,>=stealth,thick] (-5,-10.5) -- (-4.25,-10.5) -- (-3,-9.25);
\node[anchor=north] at (2,-9) {\small$P_{-1}$};
\node[anchor=north] at (-6.5,-9) {\small$P_0$};
\end{tikzpicture}
&
\begin{tikzpicture}[scale=0.18]
\node[anchor=south] at (-2,9) {\color{white}\small$P_4$};
\node[anchor=east] at (-7.7,-1) {\color{white}\small$P_{-1}$};
\node[anchor=west] at (7.7,-1) {\color{white}\small$P_{-1}$};
\node[anchor=north] at (-2,-9) {\color{white}\small$P_4$};
\node[anchor=north] at (0,8.5) {$\bA_{4,1}$};
\draw[draw=black!10,fill=black!10] (-3.1,-9) -- (-3.1,-3.9) -- (2.8,-3.9) -- (5.4,-1.3) -- (0.1,-1.3) -- (0.1,-3.1) -- (-8,-3.1) -- (-8,-8) arc (-180:-90:1) -- cycle;
\draw[fill=black] (-3.1,-9) -- (-3.1,-3.9) -- (2.8,-3.9) -- (5.4,-1.3) -- (0.1,-1.3) -- (0.1,-3.1) -- (-8,-3.1) -- (-8,-2.9) -- (-0.1,-2.9) --
(-0.1,-1.1) -- (5.6,-1.1) -- (5.6,-1.4) -- (2.9,-4.1) -- (-2.9,-4.1) -- (-2.9,-9) -- cycle;
\draw[fill=erin] (-3.9,-9) -- (-3.9,-6.9) -- (-8,-6.9) -- (-8,-7.1) -- (-4.1,-7.1) -- (-4.1,-9) -- cycle;
\draw[fill=orchid] (-5.9,-9) -- (-5.9,-4.9) -- (-8,-4.9) -- (-8,-5.1) -- (-6.1,-5.1) -- (-6.1,-9) -- cycle;
\draw[ultra thick] (-7,-9) arc (-90:-180:1) -- (-8,8) arc (180:90:1) -- (7,9) arc (90:0:1) -- (8,-8) arc (0:-90:1) -- cycle;
\draw[->,>=stealth,thick] (-1,-10.5) -- (-1.75,-10.5) -- (-3,-9.25);
\node[anchor=north] at (1,-9) {\small$P_{-1}$};
\node[anchor=north] at (-4,-9) {\small$P_0$};
\node[anchor=north] at (-6,-9) {\small$P_1$};
\node[anchor=east] at (-7.7,-7) {\small$P_2$};
\node[anchor=east] at (-7.7,-5) {\small$P_3$};
\node[anchor=east] at (-7.7,-3) {\small$P_4$};
\end{tikzpicture}
&
\begin{tikzpicture}[scale=0.18]
\node[anchor=south] at (-2,9) {\color{white}\small$P_4$};
\node[anchor=east] at (-7.7,-1) {\color{white}\small$P_{-1}$};
\node[anchor=west] at (7.7,-1) {\color{white}\small$P_{-1}$};
\node[anchor=north] at (-2,-9) {\color{white}\small$P_4$};
\node[anchor=north] at (0,8.5) {$\bA_{0,2}$};
\draw[draw=black!10,fill=black!10] (-8,-1.1) -- (-6.9,-1.1) -- (-6.9,0.9) -- (-5.9,0.9) -- (-5.9,4.8) -- (-5.8,4.9) -- (-0.1,4.9) -- (-0.1,-0.1) --
(4.8,-0.1) -- (5.7,-1) -- (5.6,-1.1) -- (-0.1,-1.1) -- (-0.1,-2.9) -- (-8,-2.9) -- cycle;
\draw[fill=black] (-8,-3.1) -- (0.1,-3.1) -- (0.1,-1.3) -- (5.7,-1.3) -- (5.9,-1.1) -- (5.9,-0.9) -- (4.9,0.1) -- (0.1,0.1) -- (0.1,5.1) --
(-5.9,5.1) -- (-6.1,4.9) -- (-6.1,1.1) -- (-7.1,1.1) -- (-7.1,-0.9) -- (-8,-0.9) -- (-8,-1.1) -- (-6.9,-1.1) -- (-6.9,0.9) --
(-5.9,0.9) -- (-5.9,4.8) -- (-5.8,4.9) -- (-0.1,4.9) -- (-0.1,-0.1) --
(4.8,-0.1) -- (5.7,-1) -- (5.6,-1.1) -- (-0.1,-1.1) -- (-0.1,-2.9) -- (-8,-2.9) -- cycle;
\draw[ultra thick] (-7,-9) arc (-90:-180:1) -- (-8,8) arc (180:90:1) -- (7,9) arc (90:0:1) -- (8,-8) arc (0:-90:1) -- cycle;
\node[anchor=east] at (-7.7,-3) {\small$P_{-1}$};
\node[anchor=east] at (-7.7,-1) {\small$P_0$};
\node[anchor=north] at (-6,-9) {\color{white}\small$P_1$};
\end{tikzpicture}
&
\begin{tikzpicture}[scale=0.18]
\node[anchor=south] at (-2,9) {\color{white}\small$P_4$};
\node[anchor=east] at (-7.7,-1) {\color{white}\small$P_{-1}$};
\node[anchor=west] at (7.7,-1) {\color{white}\small$P_{-1}$};
\node[anchor=north] at (-2,-9) {\color{white}\small$P_4$};
\node[anchor=north] at (0,8.5) {$\bA_{1,2}$};
\draw[draw=black!10,fill=black!10] (-8,0.9) -- (-7.1,0.9) -- (-7.1,-0.9) -- (-8,-0.9) -- cycle;
\draw[fill=black] (-8,-1.1) -- (-6.9,-1.1) -- (-6.9,1.1) -- (-8,1.1) -- (-8,0.9) -- (-7.1,0.9) -- (-7.1,-0.9) -- (-8,-0.9) -- cycle;
\draw[ultra thick] (-7,-9) arc (-90:-180:1) -- (-8,8) arc (180:90:1) -- (7,9) arc (90:0:1) -- (8,-8) arc (0:-90:1) -- cycle;
\node[anchor=east] at (-7.7,-1) {\small$P_{-1}$};
\node[anchor=east] at (-7.7,1) {\small$P_0$};
\node[anchor=north] at (-6,-9) {\color{white}\small$P_1$};
\end{tikzpicture}
\\
\begin{tikzpicture}[scale=0.18]
\node[anchor=south] at (-2,9) {\color{white}\small$P_4$};
\node[anchor=east] at (-7.7,-1) {\color{white}\small$P_{-1}$};
\node[anchor=west] at (7.7,-1) {\color{white}\small$P_{-1}$};
\node[anchor=north] at (-2,-9) {\color{white}\small$P_4$};
\node[anchor=south] at (0,-8.5) {$\bA_{2,2}$};
\draw[draw=black!10,fill=black!10] (-8,2.9) -- (-7.1,2.9) -- (-7.1,1.1) -- (-8,1.1) -- cycle;
\draw[fill=black] (-8,0.9) -- (-6.9,0.9) -- (-6.9,3.1) -- (-8,3.1) -- (-8,2.9) -- (-7.1,2.9) -- (-7.1,1.1) -- (-8,1.1) -- cycle;
\draw[ultra thick] (-7,-9) arc (-90:-180:1) -- (-8,8) arc (180:90:1) -- (7,9) arc (90:0:1) -- (8,-8) arc (0:-90:1) -- cycle;
\node[anchor=east] at (-7.7,1) {\small$P_{-1}$};
\node[anchor=east] at (-7.7,3) {\small$P_0$};
\node[anchor=north] at (-6,-9) {\color{white}\small$P_1$};
\end{tikzpicture}
&
\begin{tikzpicture}[scale=0.18]
\node[anchor=south] at (-2,9) {\color{white}\small$P_4$};
\node[anchor=east] at (-7.7,-1) {\color{white}\small$P_{-1}$};
\node[anchor=west] at (7.7,-1) {\color{white}\small$P_{-1}$};
\node[anchor=north] at (-2,-9) {\color{white}\small$P_4$};
\node[anchor=south] at (0,-8.5) {$\bA_{3,2}$};
\draw[draw=black!10,fill=black!10] (-3.1,9) -- (-3.1,5.9) -- (1.9,5.9) -- (1.9,0.4) -- (5.3,0.4) -- (5.8,-0.8) -- (4.9,0.1) -- (0.1,0.1) -- (0.1,5.1) --
(-5.9,5.1) -- (-6.1,4.9) -- (-6.1,1.1) -- (-6.9,1.1) -- (-6.9,3.1) -- (-8,3.1) -- (-8,8) arc (180:90:1) -- cycle;
\draw[fill=black] (-3.1,9) -- (-3.1,5.9) -- (1.9,5.9) -- (1.9,0.4) -- (5.3,0.4) -- (5.8,-0.8) -- (4.9,0.1) -- (0.1,0.1) -- (0.1,5.1) --
(-5.9,5.1) -- (-6.1,4.9) -- (-6.1,1.1) -- (-6.9,1.1) -- (-6.9,3.1) -- (-8,3.1) -- (-8,2.9) -- (-7.1,2.9) -- (-7.1,0.9) --
(-5.9,0.9) -- (-5.9,4.8) -- (-5.8,4.9) -- (-0.1,4.9) -- (-0.1,-0.1) -- (4.8,-0.1) -- (5.7,-1) -- (6,-1) -- (6,-0.7) -- (5.4,0.6) -- (2.1,0.6) --
(2.1,6.1) -- (-2.9,6.1) -- (-2.9,9) -- cycle;
\draw[fill=erin] (-8,4.9) -- (-6.1,4.9) -- (-5.9,5.1) -- (-5.9,9) -- (-6.1,9) -- (-6.1,5.2) -- (-6.2,5.1) -- (-8,5.1) -- cycle;
\draw[fill=orchid] (-3.9,9) -- (-3.9,6.9) -- (-8,6.9) -- (-8,7.1) -- (-4.1,7.1) -- (-4.1,9) -- cycle;
\draw[ultra thick] (-7,-9) arc (-90:-180:1) -- (-8,8) arc (180:90:1) -- (7,9) arc (90:0:1) -- (8,-8) arc (0:-90:1) -- cycle;
\draw[->,>=stealth,thick] (-1,10.5) -- (-1.75,10.5) -- (-3,9.25);
\node[anchor=east] at (-7.7,3) {\small$P_{-1}$};
\node[anchor=east] at (-7.7,5) {\small$P_0$};
\node[anchor=east] at (-7.7,7) {\small$P_1$};
\node[anchor=south] at (-6,9) {\small$P_2$};
\node[anchor=south] at (-4,9) {\small$P_3$};
\node[anchor=south] at (0.25,9) {\small$P_4$};
\end{tikzpicture}
&
\begin{tikzpicture}[scale=0.18]
\node[anchor=south] at (-2,9) {\color{white}\small$P_4$};
\node[anchor=east] at (-7.7,-1) {\color{white}\small$P_{-1}$};
\node[anchor=west] at (7.7,-1) {\color{white}\small$P_{-1}$};
\node[anchor=north] at (-2,-9) {\color{white}\small$P_4$};
\node[anchor=south] at (0,-8.5) {$\bA_{4,2}$};
\draw[draw=black!10,fill=black!10] (8,1.1) -- (6.1,1.1) -- (5.9,0.9) -- (5.9,-0.65) -- (5.4,0.6) -- (2.1,0.6) -- (2.1,6.1) -- (-2.9,6.1) -- (-2.9,9) -- (7,9) arc (90:0:1) -- cycle;
\draw[fill=black] (-3.1,9) -- (-3.1,5.9) -- (1.9,5.9) -- (1.9,0.4) -- (5.3,0.4) -- (5.8,-0.9) -- (6.1,-0.9) -- (6.1,0.8) -- (6.2,0.9) -- (8,0.9) --
(8,1.1) -- (6.1,1.1) -- (5.9,0.9) -- (5.9,-0.65) -- (5.4,0.6) -- (2.1,0.6) -- (2.1,6.1) -- (-2.9,6.1) -- (-2.9,9) -- cycle;
\draw[fill=erin] (-2.1,9) -- (-2.1,6.9) -- (2.1,6.9) -- (2.1,9) -- (1.9,9) -- (1.9,7.1) -- (-1.9,7.1) -- (-1.9,9) -- cycle;
\draw[fill=erin] (8,4.9) -- (6.1,4.9) -- (6.1,1.1) -- (5.9,0.9) -- (3.9,0.9) -- (3.9,7.1) -- (5.9,7.1) -- (5.9,9) -- (6.1,9) --
(6.1,6.9) -- (4.1,6.9) -- (4.1,1.1) -- (5.8,1.1) -- (5.9,1.2) -- (5.9,5.1) -- (8,5.1) -- cycle;
\draw[fill=orchid] (-0.1,9) -- (-0.1,7.9) -- (4.1,7.9) -- (4.1,9) -- (3.9,9) -- (3.9,8.1) -- (0.1,8.1) -- (0.1,9) -- cycle;
\draw[fill=orchid] (8,7.1) -- (6.9,7.1) -- (6.9,2.9) -- (8,2.9) -- (8,3.1) -- (7.1,3.1) -- (7.1,6.9) -- (8,6.9) -- cycle;
\draw[ultra thick] (-7,-9) arc (-90:-180:1) -- (-8,8) arc (180:90:1) -- (7,9) arc (90:0:1) -- (8,-8) arc (0:-90:1) -- cycle;
\node[anchor=south] at (-2,9) {\small$P_0$};
\node[anchor=south] at (0,9) {\small$P_1$};
\node[anchor=south] at (2,9) {\small$P_2$};
\node[anchor=south] at (4,9) {\small$P_3$};
\node[anchor=south] at (6,9) {\small$P_4$};
\node[anchor=west] at (7.7,7) {\small$P_5$};
\node[anchor=west] at (7.7,5) {\small$P_6$};
\node[anchor=west] at (7.7,3) {\small$P_7$};
\node[anchor=west] at (7.7,1) {\small$P_8$};
\draw[->,>=stealth,thick] (-5,10.5) -- (-4.25,10.5) -- (-3,9.25);
\node[anchor=south] at (-7,9) {\small$P_{-1}$};
\end{tikzpicture}
&
\end{tabular}
}
\caption{Splitting a partial configuration of type $2$ into $2 \times 5 + 1$ partial configurations}
\end{center}
\label{fig:4}
\end{figure}
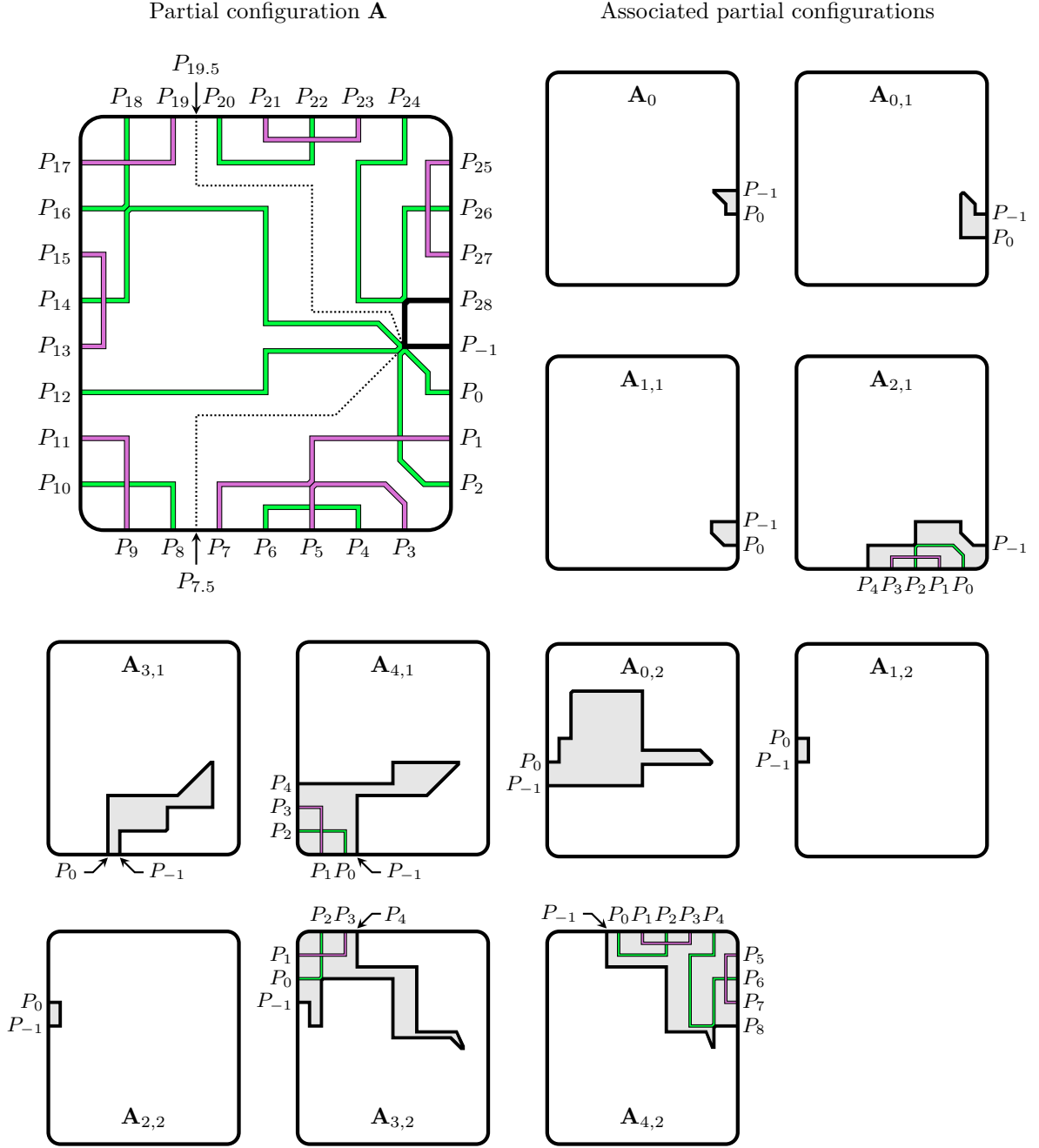
The following result follows immediately.
\begin{theorem}
\label{thm:19}
The bivariate generating series $\bN_1(x,y)$ and $\bN_2(x,y)$ are solutions of:
\[\bN_1 = 1 + y \bN_2^4 \text{ and } (1 - \bN_2 + y \bN_2^4)(1 + y \bN_2^4 - x^2 y \bN_2^5) + x y \bN_2^6 = 0.\]
In particular, $\bN_1$ is algebraic, and therefore holonomic.
\end{theorem}
\begin{proof}
Lemmas~\ref{lem:15},~\ref{lem:16} and~\ref{lem:17} state explicitly the equalities
\[\bN_1 = 1 + y \bN_2^4 \text{, } \bN_2 = \bN_1 \bN_3 \text{ and } \bN_2 = \bN_1 + \frac{x y \bN_2^5 \bN_3}{1 - x^2 y \bN_2^4 \bN_3}.\]
The first two equalities allow us to express $\bN_1$ and $\bN_3$ in terms of $\bN_2$, thanks to which the third equality is rewritten as
\[(1 - \bN_2 + y \bN_2^4)(1 + y \bN_2^4 - x^2 y \bN_2^5) + x y \bN_2^6 = 0,\]
which completes the proof.
\end{proof}

\begin{remark}
By adapting the above arguments, we might directly obtain similar results for different, piecewise-affine notions of codimensions.
In general, defining the local codimension  at a critical point (i.e. of a meeting point) of valency $2k$ of the graph as an integer $c(k)$,
Lemmas~\ref{lem:15} and~\ref{lem:16} are left unchanged, while Lemma~\ref{lem:17} indicates that
$\bN_2 = \bN_1 + \sum_{k \geq 1} x^{c(k+1)} y^k \bN_2^{4k+1} \bN_3^k$, from which a variant of Theorem~\ref{thm:19} follows as well.

For example, if we set $c(k) = k - \mathbf{1}_{k = 2}$, then Theorem~\ref{thm:19} states that
\[\bN_1 = 1 + y \bN_2^4 \text{ and } (1 - \bN_2 + 2 y \bN_2^4 - y \bN_2^5 + x y \bN_2^6 + y^2 \bN_2^8)(1 + y \bN_2^4 - x y \bN_2^5) + x^3 y^2 \bN_2^9 = 0.\]
\end{remark}

\bigskip

\paragraph{Algorithmic consequences}

Since $\bN_1$ is algebraic, it is known that $\bN_1$ is a solution of a linear partial-differential equation,
which can be made explicit. From such an equation, it is then possible to derive a linear recurrence equation,
whose coefficients are rational fractions in $c$ and $d$, and that is satisfied by the family $(\#N_1(c,d))_{c,d \geq 0}$.
However, it follows from Theorem~\ref{thm:19} that the minimal polynomial of $\bN_1(x,y)$ is
the following 60-term coefficient polynomial in $z$:
\begin{align*}P(x,y,z) = &
 -x^4 - 4 x^5 - 6 x^6 - 4 x^7 - x^8 + 6 x^4 z + 24 x^5 z + 36 x^6 z +
 24 x^7 z + 6 x^8 z - 15 x^4 z^2 - 60 x^5 z^2 \\
 & - 90 x^6 z^2 -
 60 x^7 z^2 - 15 x^8 z^2 + 20 x^4 z^3 + 80 x^5 z^3 + 120 x^6 z^3 +
 80 x^7 z^3 + 20 x^8 z^3 - 15 x^4 z^4 \\
 & - 60 x^5 z^4 - 90 x^6 z^4 -
 60 x^7 z^4 - 15 x^8 z^4 - y z^4 - 4 x y z^4 - 2 x^2 y z^4 +
 4 x^3 y z^4 - 4 x^5 y z^4 - x^8 y z^4 \\
 & + 6 x^4 z^5 + 24 x^5 z^5 +
 36 x^6 z^5 + 24 x^7 z^5 + 6 x^8 z^5 + y z^5 + 8 x y z^5 +
 6 x^2 y z^5 - 12 x^3 y z^5 + 16 x^5 y z^5 \\
 & + 5 x^8 y z^5 - x^4 z^6 -
 4 x^5 z^6 - 6 x^6 z^6 - 4 x^7 z^6 - x^8 z^6 - 4 x y z^6 -
 6 x^2 y z^6 + 12 x^3 y z^6 - 24 x^5 y z^6 \\
 & - 10 x^8 y z^6 +
 2 x^2 y z^7 - 4 x^3 y z^7 + 16 x^5 y z^7 + 10 x^8 y z^7 -
 4 x^5 y z^8 - 5 x^8 y z^8 - y^2 z^8 + x^8 y z^9.
\end{align*}

The complexity of this polynomial led us to prefer stating Theorem~\ref{thm:19} as above,
by involving two simple algebraic equations instead of a single, complicated one.
In addition, even the simple equation satisfied by $\bN_2$, from which we might have derived
insights about $\bN_1$, did not allow us to obtain tractable linear partial-differential equations.

Hence, we prefer using a seemingly naive algorithm for computing the integers $\#N_1(c,d)$:
this approach still provides us with excellent results in practice.
It consists in first computing recursively the coefficients $\#N_2(c,d)$ of the generating function
$\bN_2$: more precisely, the integer that we store in the entry $\mathbf{N}(c_1,d_1,e)$ of the array $\mathbf{N}$
is the coefficient $[x^{c_1},y^{d_1}]$ of the generating function $\bN_2(x,y)^e$.

The computational cost of this algorithm is the following.
First, and anticipating on Lemma~\ref{lem:29} below,
entries of the array $\mathbf{N}$ are non-negative integers not greater than
$\mathcal{O}(26^d)$, and we know that $\#N_1(c,d) = 0$ whenever $c \geq 2d$.
Furthermore, it is clear that Algorithm~\ref{algo:1} performs
$\mathcal{O}(c^2 d^2)$ arithmetic operations on such integers.
Consequently, and even using naive multiplication algorithms,
which would therefore be performed in $\mathcal{O}(d^2)$ bit operations,
the time complexity of our algorithm is $\mathcal{O}(\min\{c^2 d^4,d^6\})$.

\begin{algorithm}[h]
\label{algo:1}
\algsetup{linenosize=\normalfont}\small
 \KwIn{codimension $c$ and degree $d$}
 \KwOut{integer $\#N_1(c,d)$}
 \If{$c \geq 2d$}{
 \Return{$0$}
 }
 $\mathbf{N} \leftarrow \text{ new } (c+1) \times d \times 10 \text{ array}$\;
 \For{$c_1 = 0,\ldots,c$}{
  \For{$d_1 = 0,\ldots,d-1$} {
  \eIf{$c_1 = 0$ and $d_1 = 0$}{
   $\mathbf{N}(c_1,d_1,0) \leftarrow 1$\;
   }{$\mathbf{N}(c_1,d_1,0) \leftarrow 0$\;
   }
   $\mathbf{N}(c_1,d_1,1) \leftarrow \mathbf{N}(c_1,d_1,0) + 2 \mathbf{N}(c_1,d_1-1,4) - \mathbf{N}(c_1,d_1-1,5) -
   \mathbf{N}(c_1-2,d_1-1,5)$ $~+~\mathbf{N}(c_1-1,d_1-1,6) + \mathbf{N}(c_1-2,d_1-1,6) + \mathbf{N}(c_1,d_1-2,8) - \mathbf{N}(c_1-2,d_1-2,9)$\;
   \For{$e=2,\ldots,9$}{
    $\mathbf{N}(c_1,d_1,e) \leftarrow 0$\;
    \For{$c_2 = 0,\ldots,c_1$}{
      \For{$d_2 = 0,\ldots,d_1$}{
	$\mathbf{N}(c_1,d_1,e) \leftarrow \mathbf{N}(c_1,d_1,e) + \mathbf{N}(c_2,d_2,1) \mathbf{N}(c_1-c_2,d_1-d_2,e-1)$\;
      }
    }
   }
  }
 }
 \Return{$\mathbf{N}(c,d-1,4)$}
 
 \medskip
 
 \caption{Counting configurations of codimension $c$ and degree $d$ up to equivalence}
\end{algorithm}

\medskip

\begin{lemma}
\label{lem:29}
Let $c$ and $d$ be two non-negative integers. If $c \geq 2d$, then $\#N_1(c,d) = 0$.
Furthermore, there exists an integer $\mathbf{K}$, independent of $c$ and $d$,
such that $\#N_1(c,d) \leq \mathbf{K} 26^d$.
\end{lemma}

\begin{proof}
First, since equivalence classes of polynomials with a given configuration
are subvarieties of a real vector space of dimension $(2d-1)$,
it is clear that no configuration can be of codimension $c \geq 2d$.
This proves the first part of Lemma~\ref{lem:29}.

The second part of Lemma~\ref{lem:29} is proved as follows.
Considering the generating series $\bN_1(x,y)$ for $x = 1$, we observe that
the sum $\sum_{c \geq 0} \#N_1(c,d)$ is equal to the $d$-th coefficient $[y^d]\bN_1(1,y)$.
Note that the series $\bN_1(1,y)$ is cancelled by the polynomial $Q(y,z) = P(1,y,z)$, i.e.
\[Q(y,z) = -16 + 96 z - 240 z^2 + 320 z^3 - 240 z^4 - 8 y z^4 + 96 z^5 +  24 y z^5 - 16 z^6 - 32 y z^6 + 24 y z^7 - 9 y z^8 - y^2 z^8 + y z^9,\]
It follows that the coefficients $[y^d]\bN_1(1,y)$ are of the order of magnitude of $y_0^{-d}$, where
$y_0$ is the smallest positive value such that the polynomials $Q(y_0,z)$ and $\partial_z Q(y_0,z)$ have a common root $z_0$.
More precisely, we find that $y_0$ is also the smallest positive root of the irreducible polynomial
\[R(y) = -84375 + 1620000 y + 12241152 y^2 + 21528576 y^3 + 1048576 y^4,\]
i.e. $y_0^{-1} \approx 25.327$, and that $[y^d]\bN_1(1,y) \sim y_0^{-1} d$ when $d \to +\infty$.
This completes the proof.
\end{proof}

\section{Counting simple configurations}
\label{section:4}
Theorem~\ref{thm:19} provides us with polynomial equations,
defining implicitly the series $\bN_1$ and $\bN_2$.
However, the size of these equations (including their degree)
do not allow extracting simple general formulas
for the coefficients $\#N_1(c,d)$.
Therefore, we investigate more closely a restricted class of configurations,
which we will subsequently be able to count efficiently.
\begin{definition}
\label{def:20}
A configuration $\bA$ is said to be \emph{simple} if every connected component of $\bA$
contains at most one meeting point, and if exactly two lines touch each other at that point.
\end{definition}
In the sequel, we denote by $N_4(c,d)$ the number of (equivalence classes of)
simple configurations of codimension $c$ and degree $d$,
and let $\bN_4(x,y) = \sum_{c,d \geq 0} \#N_4(c,d) x^c y^d$ be the associated
bivariate generating function.
\begin{lemma}
\label{lem:21}
We have $\bN_4 = 1 + y \bN_4^4 + 4 x y^2 \bN_4^8$.
\end{lemma}
\begin{proof}
Let us use a bijective proof again.
Let $\bA$ be a configuration of degree $d \geq 1$, with canonical splitting $\bS$.
Let $L^e$ be the even line (in $\bS$) with endpoint $P_0$, and let $L^o$ be the odd line that crosses $L^e$.
We say that $\bA$ has type 1 if the union $L^e \cup L^o$ contains no meeting point of $\bA$,
and that $\bA$ has type 2 otherwise.
For $k = 1, 2$, let $N_{4,k}$ be the set of configurations of type $k$.
We first design a bijection $\varphi_1 : N_{4,1} \mapsto N_4^4$
such that $\varphi_1(\bA) = (\bA_i)_{0 \leq i \leq 3}$, where
$\deg \bA = \sum_{i=0}^3 \deg \bA_i + 1$ and
$\codim{\bA} = \sum_{i=0}^3 \codim{\bA_i}$.

This bijection is very similar to that of Lemma~\ref{lem:15}:
the set $\bU \setminus (L^e \cup L^o)$ is formed of $4$ connected components $\bC_0, \ldots, \bC_3$ and,
for $i = 0, 1, 2, 3$, up to deleting the points $P_x$ that do not belong to $\partial \bC_i$
and to renaming those that belong to $\partial \bC_i$, we observe that $\bA \cap \bC_i$
is a simple configuration. Denoting this configuration by $\bA_i$,
we set $\varphi_1(\bA) = (\bA_i)_{0 \leq i \leq 3}$.
Observing that $\varphi_1$ is bijective and satisfies the above equalities
(involving codimensions and degrees) is then straightforward.

In the same vein, we design another bijection $\varphi_2 : N_{4,2} \mapsto \{0,1,2,3\} \times N_4^8$
such that $\varphi_2(\bA) = (k,(\bA_i)_{0 \leq i \leq 7})$, where
$k \in \{0,1,2,3\}$, $\deg \bA = \sum_{i=0}^7 \deg \bA_i + 2$ and
$\codim{\bA} = \sum_{i=0}^7 \codim{\bA_i} + 1$.
Let $X$ be the point at which $L^e$ crosses $L^o$.
This point splits both $L^e$ and $L^o$ into four half-lines
$L_0, \ldots, L_3$,
with respective endpoints $P_{a_0}$, $P_{a_1}$, $P_{a_2}$ and $P_{a_3}$
(with $0 = a_0 < a_1 < a_2 < a_3 < 4d$ and $a_i \equiv i \pmod{4}$).

One of these half-lines (which we denote $L_u$ in the sequel) contains a meeting point $Y$, at which
it touches another line $L'_0$ of $\bA$. The line $L'_0$ itself is crossed by
a line $L'_1$, and none of $L'_0$ or $L'_1$ contains a meeting point different from $Y$.
Hence, the set $\bU \setminus (L^e \cup L^o \cup L'_0 \cup L'_1)$ is formed of
$8$ connected components $\bC_0,\ldots,\bC_7$.
Up to reordering these components, we assume that,
for all $i$, the border $\partial \bC_i \cap \partial \bU$ is an arc of circle $[P_{a_i},P_{a_{i+1}}]$,
with $0 = a_0 < a_1 < \ldots < a_9 = 4d$.
Furthermore, we can again check that, up to deleting and renaming points $P_x$,
$\bA \cap \bC_i$ is a simple configuration for all $i$, and we denote it by $\bA_i$.

Hence, we set $\varphi_2(\bA) = (u,(\bA_i)_{0\leq i\leq 7})$,
where we recall that $L_u$ was the line to which belongs the point $Y$.
Again, $\varphi_2$ is bijective and satisfies the above requirement about
degrees and codimensions.
Observing that the simple configuration of degree $0$
is the unique simple configuration that does not belong to $N_{4,1} \cup N_{4,2}$
completes the proof.
\end{proof}
\begin{theorem}
\label{thm:22}
For all integers $c, d \geq 0$, we have
\[\#N_4(c,d) = \mathbf{1}_{d \geq 2c} \frac{4^c}{c + 3 d + 1} \binom{4 d}{c,d-2c,c + 3 d}.\]
In particular, for a fixed value of $c$ and when $d \to +\infty$, we have
\[\#N_4(c,d) \sim \sqrt{\frac{2}{27 \pi}} \fracexp{4e}{3}{c} \fracexp{4^4}{3^3}{d} \frac{d^{c-3/2}}{c!}.\]
\end{theorem}
\begin{proof}
Consider the function $G : (u,z) \mapsto \bN_4((2 z)^{-2} u,z)$. Setting $u = (2 z)^2 t$, we have
$1 - G + z G^4 - u G^8 = 0$.
Consider also the bivariate functions $h : (u,z) \mapsto 1+u+z$, $g_u : (u,z) \mapsto h(u,z)^8$ and $g_z : (u,z) \mapsto h(u,z)^4$,
and let $G_u$ and $G_z$ be the solutions of the equations
\[G_u(u,z) = u g_u(G_u(u,z),G_z(u,z)) \text{ and } G_z(u,z) = z g_z(G_u(u,z),G_z(u,z)).\]
Observe that the function $\mathbf{G} : (u,z) \mapsto h(G_u(u,z),G_z(u,z))$ is solution of the equation
\[\mathbf{G}(u,z) = 1 + G_u(u,z) + G_z(u,z) = 1 + u \mathbf{G}(u,z)^8 + z \mathbf{G}(u,z)^4,\]
i.e. that $G = \mathbf{G}$.
The bivariate Lagrange inversion formula states, for all integers $k, \ell \geq 0$, that
\begin{align*}
[u^k,z^\ell] h(G_u,G_z) & = \frac{1}{k \ell} [u^{k-1},z^{\ell-1}] \left((\partial_u h) (\partial_z g_u^k) g_z^\ell + (\partial_z h) (\partial_u g_z^\ell) g_u^k + (\partial_u \partial_z h) g_u^k g_z^\ell\right) \\
& = \frac{1}{k \ell} [u^{k-1},z^{\ell-1}] (8 k + 4 \ell) (1+u+z)^{8 k + 4 \ell - 1} \\
& = \frac{8 k + 4 \ell}{k \ell} \binom{8 k + 4 \ell - 1}{k-1, \ell-1, 7 k + 3 \ell + 1} \\
& = \frac{1}{7 k + 3 \ell + 1} \binom{8 k + 4 \ell}{k,\ell,7 k + 3 \ell}.
\end{align*}
It follows that
\begin{align*}
\bN_4(t,z) & = \mathbf{G}(4 z^2 t,z) \\
& = \sum_{k,\ell \geq 0} \frac{1}{7 k + 3 \ell + 1} \binom{8 k + 4 \ell}{k,\ell,7 k + 3 \ell} (4 z^2 t)^k z^\ell \\
& = \sum_{k,m \geq 0} \mathbf{1}_{m \geq 2k} \frac{4^k}{k + 3 m + 1} \binom{4 m}{k,m-2k,k + 3 m} t^k z^m
\end{align*}
This proves the equality
\[\#N_4(c,d) = \mathbf{1}_{d \geq 2c} \frac{4^c}{c + 3 d + 1} \binom{4 d}{c,d-2c,c + 3 d}.\]
Using the Stirling formula when $d \to +\infty$ then provides the asymptotic estimation of Theorem~\ref{thm:23}.
\end{proof}
Since all configurations of codimension $c \leq 1$ are simple, Theorem~\ref{thm:23} follows immediately.
\begin{theorem}
\label{thm:23}
For all integers $d \geq 0$, we have
\[\#N_1(0,d) = \frac{1}{4d+1}\binom{4d+1}{d} \text{ and } \#N_1(1,d) = \mathbf{1}_{d \geq 2} 4 \binom{4d}{d-2}.\]
\end{theorem}
\section{Asymptotic bounds for the number of configurations}
\label{section:5}
In spite of Theorem~\ref{thm:22}, providing explicit formulas for the coefficients $\#N_1(c,d)$
seems out of reach. Therefore, we focus here on asymptotical evaluations of $\#N_1(c,d)$ for fixed values of $c$, when $d \to +\infty$.
\begin{definition}
\label{def:24}
Let $\bA$ be a configuration, and let $\bS$ be the canonical splitting of $\bA$.
A connected component of $\bA$ is said to be \emph{friendly} if it contains a meeting point,
and a line in $\bS$ is also said to be \emph{friendly} if it belongs to a friendly component of $\bA$.
Furthermore, we call \emph{width} of $\bA$ the number of friendly components of $\bA$ and
\emph{size} of $\bA$ the number of crossing points belonging to a friendly component of $\bA$;
we denote them respectively by $w(\bA)$ and $s(\bA)$.
Finally, the \emph{combinatorial type} of $\bA$, denoted by $ct(\bA)$, is the configuration of degree $s(\bA)$ obtained
by deleting all non-friendly lines of $\bA$ and the points of $\partial \bU$ that they join, and by renaming the remaining points
from $P_0$ to $P_{4 s(\bA)-1}$ without changing their relative orders.
If $\bA = ct(\bA)$, then $\bA$ is said to be a \emph{combinatorial configuration}.
\end{definition}
\begin{lemma}
\label{lem:25}
Let $\bA$ be a configuration with codimension $c$. We have $w(\bA) \leq c$ and $2 w(\bA) \leq s(\bA) \leq c + w(\bA)$.
\end{lemma}
\begin{proof}
We draw a graph $\mathcal{G}$ as follows. Let $\mathbf{P}$ be the set of meeting points or crossing points of $\bA$
and of integer points (lying at the border of the circle).
We put a vertex on every point in $\mathbf{P}$, and we draw an edge between two points $P_1, P_2 \in \mathbf{P}$
if and only if there exists a line of $\bA$ that contains both $P_1$ and $P_2$,
and contains no other point of $\mathbf{P}$ between $P_1$ and $P_2$.
There exists no cycle in $\bA$, hence $\mathcal{G}$ contains no cycle as well, i.e. it is a forest.

Let $m$ be the number of meeting points in $\bA$. Observe that $w(\bA) \leq m \leq c$,
from which follows the first inequality of Lemma~\ref{lem:25}.
The graph $\mathcal{G}$ contains $4d$ integer-point vertices (of degree $1$),
$d$ meeting-point vertices (of degree $4$) and $m$ meeting-point vertices;
every meeting point $P$ of codimension $\kappa$ is associated with a vertex of degree
at most $2(\kappa+1)$ (with equality if and only if $\kappa = 1$).
Hence, $\mathcal{G}$ contains at most $4d+m+c$ edges.
At the same time, $\mathcal{G}$ is a forest made of 
$d-s(\bA) + w(\bA)$ disjoint trees with a total of $5d+m$ nodes,
hence it contains $4d+m+s(\bA) - w(\bA)$ edges.
It follows that
\[4d+m+s(\bA) - w(\bA) \leq 4 d + m+c,\]
i.e. that $s(\bA) \leq c + w(\bA)$. Since every friendly connected component of $\bA$
contains at least two crossing point, the second inquality of Lemma~\ref{lem:25} follows.
\end{proof}
\begin{corollary}
\label{cor:26}
Every configuration has the same codimension as its combinatorial type.
Moreover, for every integer $c \geq 0$, there exists finitely many combinatorial configurations
of codimension $c$.
\end{corollary}
In the remainder of this section, we fix two integers $c, d \geq 0$,
and we denote by $\Comb(c)$ the set of all combinatorial configurations of codimension $c$;
For all $\bB \in \Comb(c)$, then we also denote by $N_1(\bB)$ the set $\{\bA \in N_1 \mid ct(\bA) = \bB\}$
and by $N_1(\bB,d)$ the set $\{\bA \in N_1 \mid ct(\bA) = \bB\}$.
Corollary~\ref{cor:26} states that $\Comb(c)$ is finite, and that $N_1(c,d)$ is the disjoint union
of the sets $N_1(\bB,d)$ for $\bB \in \Comb(c)$.
\begin{lemma}
\label{lem:27}
Let $\bB$ be a combinatorial configuration of codimension $c$, and let $\Delta = 1 + 4s(\bB) - w(\bB)$.
For all integers $d \geq 0$, we have
\[\#N_1(\bB,d) \leq (4d)^{w(\bB)} \sum_{d_0+\ldots+d_{\Delta-1} = d - s(\bB)} \prod_{i=0}^{\Delta-1} \#N_1(0,d_i).\]
\end{lemma}
\begin{proof}
Consider some configuration $\bA \in N_1(\bB,d)$.
Let $\bF$ be the union of all the friendly connected components of $\bA$.
This union consists of $w(\bB)$ cycle-free connected components,
which touch a total of $4 s(\bB)$ points on $\partial \bU$.
Hence, the set $\bU \setminus \bF$ is made of $\Delta$ connected components.
Let these components be $\bC_0,\ldots,\bC_{\Delta-1}$.
For every component $\bC_i$, there exists a minimal integer $a_i$ and a
maximal integer $b_i$ such that the arcs of circle $[P_{a_i},P_{a_i+1}]$ and
$[P_{b_i-1},P_{b_i}]$ belong to $\partial \bC_i$.

Without loss of generality, we assume that $0 \leq a_0 < a_1 < \ldots < a_{\Delta-1} < 4d \leq a_\Delta$
(with the convention that $a_\Delta = a_0 + 4d$).
Then, let $d_i$ be the number of crossing points contained in $\bC_i$.
Up to deleting those points $P_x$ such that $[P_{x-1},P_{x+1}]$ is not contained in $\partial \bC_i$
and renaming the remaining $4d_i$ points from $P_0$ to $P_{4d_i-1}$ (so that two points $P_x$ and $P_y$ with $x < y$
be renamed $P_{x'}$ and $P_{y'}$ with $x' < y'$), we observe that $\bA \cap \bC_i$
is a flat configuration (i.e. of codimension $0$) and degree $d_i$.
Furthermore, for every connected component $F \in \bF$,
we denote by $\iota(F)$ be the minimal integer $x$ such that $P_x \in F$.

Observe that knowing the configuration $\bB$, the configurations $\bA \cap \bC_i$ and the integers $d_i$ and $\iota(F)$
is enough to determine unambiguously the configuration $\bA$ itself.
There are $w(\bB)$ integers $\iota(F)$ to know, and we have $0 \leq \iota(F) < 4d$ for each of them.
Moreover, observe that $\sum_{i=0}^{\Delta-1} d_i = d - s(\bB)$.
Lemma~\ref{lem:27} follows.
\end{proof}
\begin{lemma}
\label{lem:28}
For all integers $a \geq 1$ and $b \geq 0$, we have
\[\sum_{d_1+\ldots+d_a = b} \prod_{i=1}^a \#N_1(0,d_i) = \frac{a}{4b+a} \binom{4b+a}{b}.\]
\end{lemma}
\begin{proof}
Let $\bM(y)$ be the generating function $\bM(y) = \sum_{b \geq 0} m_b y^b$, where we set
$m_b = \sum_{d_1+\ldots+d_a = b} \prod_{i=1}^a \#N_1(0,d_i)$.
Recall, using Lemma~\ref{lem:21}, that $\bN_1(0,y) = \bN_4(0,y) = 1 + y \bN_4(0,y)^4 = 1 + y \bN_1(0,y)^4$.
Since $\bM(y) = \bN_1(0,y)^a$, it follows that $\bM$ is the only solution of the equation
$\bM(y) = (1 + y \bM(y)^{4/a})^a$ with non-negative coefficients,
i.e. is the Fuss-Catalan generating function
\[\bM(y) = \bB_{4,a}(y) = \sum_{b \geq 0} \frac{a}{4b+a} \binom{4b+a}{d} y^b.\]
This completes the proof of Lemma~\ref{lem:28}.
\end{proof}
\begin{theorem}
\label{thm:29}
Let $c \geq 0$ be a fixed integer. When $d \to +\infty$, we have
\[\#N_1(c,d) \sim \sqrt{\frac{2}{27 \pi}} \fracexp{4e}{3}{c} \fracexp{4^4}{3^3}{d} \frac{d^{c-3/2}}{c!}.\]
\end{theorem}
\begin{proof}
First, observe that $\#N_1(c,d) \geq \#N_4(c,d)$ for all $d \geq 0$.
Since we want to prove that $\#N_1(c,d) \sim \#N_4(c,d)$ when $d \to +\infty$,
it remains to show that $\#N_1(\bB,d) = o(\#N_4(c,d))$
for all non-flat combinatorial configurations $\bB \in \Comb(c)$.
Hence, let $\bB$ be such a combinatorial configuration, and let $\Delta =  1 + 4s(\bB) - w(\bB)$.
Since $w(\bB) \leq c-1$, it follows from Lemmas~\ref{lem:27} and~\ref{lem:28} that
\[\#N_1(\bB,d) \leq (4d)^{c-1} \frac{\Delta}{4(d - s(\bB))+\Delta} \binom{4(d - s(\bB))+\Delta}{d - s(\bB)}.\]
The Stirling formula therefore proves that 
\[\#N_1(\bB,d) = \mathcal{O}\left(d^{c-5/2} \frac{4^{4d}}{3^{3d}} \right)\]
when $d \to +\infty$, which completes the proof.
\end{proof}

\section{Conclusion}
The number of configurations for a fixed degree and codimension $c$ follow from our results on generating functions.  
We therefore answer a question of N. A'Campo, concerning the number and the growth of his bi-colored forests.  
These combinatorial tools are very useful for the understanding in how many sub-parts the space of complex polynomials can be decomposed, according to our criterion for the decomposition.

\end{document}